\documentclass[11pt,letterpaper,reqno,oneside]{amsart}
\usepackage{amssymb}

\usepackage{geometry}

\usepackage{graphicx}
\usepackage{amsmath}
\usepackage{tikz}
\usepackage{hyperref}
\usepackage{microtype}
\hypersetup{
	pdfstartview={XYZ null null 1.00}, 
	pdfpagemode=UseNone, 
	colorlinks,
	breaklinks, 
	linkcolor=blue,
	urlcolor=blue, 
	anchorcolor=blue,
	citecolor=blue
}
\usepackage[capitalise]{cleveref}
\usepackage[shortlabels]{enumitem}

\newtheorem{theorem}{Theorem}[section]
\newtheorem{question}[theorem]{Question}
\newtheorem{lemma}[theorem]{Lemma}
\newtheorem{prop}[theorem]{Proposition}

\newtheorem{conj}[theorem]{Conjecture}
\newtheorem{claim}[theorem]{Claim}
\newtheorem{subclaim}[theorem]{Sub-claim}

\theoremstyle{definition}
\newtheorem{definition}[theorem]{Definition}
\newtheorem{remark}[theorem]{Remark}

\newcommand{\prodr}{\overrightarrow{\prod}}
\newcommand{\prodl}{\overleftarrow{\prod}}
\newcommand{\IP}{\overleftarrow{IP_0}}

\newcommand{\ep}{\varepsilon}

\title{Monochromatic non-commuting products}
\author{Matt Bowen}
\address{University of Warwick}
\email{Matt.A.Bowen@warwick.ac.uk}
\thanks{During the work on this article the author was supported by Ben Green's Simons Investigator Grant number 376201}
\date{October 2026}

\begin{document}

\begin{abstract}
    We show that any finite coloring of an amenable group contains `many' monochromatic sets of the form $\{x,y,xy,yx\},$ and natural extensions with more variables.  This gives the first combinatorial proof of Bergelson and McCutcheon's non-commutative Schur theorem, together with extensions.  Our main new tool is the introduction of what we call `quasirandom colorings,' a condition that is automatically satisfied by colorings of quasirandom groups, and a reduction to this case. 
\end{abstract}

\maketitle

\section{Introduction}

In this paper we will be concerned with the following conjectures of Bergelson and McCutcheon.  Given a group $G$, let $A=\{g\in G: [G:C(g)]<\infty\},$ where $C(g)$ is the centralizer of $g$ in $G$. In the following and throughout, a group is called \textbf{sufficiently non-commutative} if $[G:A]=\infty$.  In amenable groups, this implies that the probability that two elements commute is zero and Tointon \cite{tointon2020commuting} has shown that any finitely generated amenable group which is not virtually abelian is sufficiently non-commutative.

\begin{conj}[\cite{bergelson2009selection}, Conjecture 2.3] \label{conj1}

     Suppose that a sufficiently non-commutative group is finitely colored.  There is a monochromatic set of the form $\{x,y,xy,yx\}$ with $xy\neq yx.$ 
\end{conj}

\vspace{1mm}

\begin{question}[\cite{bergelson1998recurrence}, the paragraph before Theorem 3.4]\label{conj2}

      Suppose that a sufficiently non-commutative group is finitely colored.  Is there a sequence $x_1,...,x_k$ with all possible products (in any order) monochromatic and distinct?
\end{question}

\vspace{1mm}

While the above conjectures are still quite open for general groups (and the free group on two generators in particular), a number of special cases have been closely analyzed and resolved.  Namely, Bergelson and McCutcheon \cite{bergelson1998recurrence,bergelson2007central} were able to prove Conjecture \ref{conj1} for colorings of infinite amenable groups.  For finite groups (which are always amenable), Sanders \cite{sanders2019schur} later proved an analogous result, extending work of Bergelson and Tao \cite{bergelson2014multiple}, who had proven a stronger density result in the more restrictive setting of finite quasirandom groups.  Results similar to those of Bergelson and Tao, with good quantitative bounds, were also obtained by Austin \cite{austin2015quantitative} and Tserunyan \cite{tserunyan2023mixing}.  For Question \ref{conj2}, there appear to be no special cases where even the $k=3$ case is fully known, with the main progress so far due to Bergelson, Christopherson, Robertson, and Zorin-Kranich, who proved a density version of a weaker form of the conjecture for infinite amenable quasirandom groups \cite{bergelson2016finite}. 

A peculiar aspect shared by all of the previously mentioned papers is that they deduce their coloring results from a `non-commutative Roth theorem,' i.e., a density result which controls the pattern $\{x,xy,yx\}.$ Moreover, it appears that there was no previously known direct `coloring' proof of any special case of Conjecture \ref{conj1}.  This is in contrast with most classical results in Ramsey theory, where coloring results are typically much easier than their density counterparts and are often known well in advance of them.  It has also served as a major bottleneck in extending progress on Conjecture \ref{conj1} and Question \ref{conj2} beyond the $k=2$ case (where non-commutative measure recurrence results seem to quickly become elusive) and the setting of amenable groups (where density stops being a viable tool altogether). 

In this paper we give a direct coloring proof of the amenable case of Conjecture \ref{conj1} which extends to handle some of the previously unknown subcases of Question \ref{conj2}.    

\begin{theorem}\label{thm:main}
    Suppose that $G$ is an infinite amenable group with $G$-invariant finitely additive probability measure $\mu.$  In any finite coloring of $G$ we have $$\mu(\{y: \mu(\{x: \{x,y,xy,yx\} \textnormal{ is monochromatic}\})>0\})>0.$$

    More generally, for any $N$ we can find many\footnote{See Definition \ref{def: many}.  Also, for a subset $I\subseteq \{m,...,n\},$ the notation $\prodr_{a\in I}$ is used to denote the product in increasing order of indices. } tuples $(x_N,...,x_0)$ for which the set $$\{x_i\cdot x_{i+1}\cdot...\cdot x_j, \overrightarrow{\prod_{a\in I}}x_a\cdot x_0\cdot x_1\cdot...\cdot x_j: 0\leq i\leq j\leq N, I\subseteq \{j+1,...,N\}\}$$ is monochromatic.
\end{theorem}

Note that one wishes to find a monochromatic set of size $15$ in the $k=3$ case of Question \ref{conj2}, corresponding to all non-repeating products in any order of terms $x,y,z.$  In this case, the second part of Theorem \ref{thm:main} controls $10$ monochromatic products, with the missing terms being $xz,zy,xzy,yxz,zyx.$

Up to keeping track of extra parameters, one can also obtain lower bounds on the measure of the sets in Theorem \ref{thm:main} and similarly deduce a version of the theorem for finite groups.  The bounds this would give are fairly poor and most likely far from optimal, so we will omit the details for ease of exposition.


Using a more basic form of the same idea, we also make the following progress on the non-amenable case of Conjecture \ref{conj1}.

\begin{prop}\label{prop: f2}
    In any finite coloring of a group we have $$\{y: \{x: \{x,xy,yx\} \textnormal{ is monochromatic}\} \textnormal{ is piecewise syndetic} \}$$
    is piecewise syndetic.
\end{prop}

A similar result was proven by Bergelson and Hindman in \cite{bergelson1992some} using idempotent ultrafilters.  The argument we give here appears to be the first combinatorial proof of this fact and primarily relies on a `piecewise syndetic pigeonhole principle' (see Proposition \ref{pws php})  which says that any piecewise syndetic subset of an arbitrary group contains `many' sets of the form $\{x,xy\}.$  Unfortunately the notion of `many' that appears here is less robust than the one appearing in the `density pigeonhole principle' for amenable groups (see Lemma \ref{php}), and this is the main obstacle in extending our argument to control the color of $y$ as well.

\subsection{Techniques}

\hspace{3mm}

The basic idea behind the proof of Theorem \ref{thm:main} was inspired by the author's recent work with Sabok \cite{bowen2025monochromatic,bowen.sabok} on finding monochromatic sets of the form $\{x,y,xy,x+y\}$ in colorings of $\mathbb{N}$ and $\mathbb{Q},$ although the execution of the idea is quite different in the present paper.  The starting point of these results is that it is fairly easy, using standard dynamical color focusing arguments, to prove a weaker result controlling the color of all terms excluding $y.$  From here one then argues that if the coloring is `sufficiently random' then this color focusing argument can be upgraded to also control the color of $y.$  Otherwise, the coloring must be fairly structured and the desired conclusion can be deduced by exploiting this structure.

The primary difference between the arguments in the present paper and those in \cite{bowen2025monochromatic,bowen.sabok} comes from which definition of `sufficiently random' is useful.  In \cite{bowen2025monochromatic,bowen.sabok} the useful structure-versus-randomness dichotomy came from the dichotomy between thick and syndetic sets.  While the thick case of the arguments from those papers is applicable to the present problem, it seems that syndeticity is a much less useful condition in the non-commutative setting.

Instead, the notion of randomness that we develop in the present paper is related to the properties of colorings of what are known as \textbf{quasirandom} groups in the finite setting \cite{gowers2008quasirandom} and \textbf{weak mixing} \cite{bergelson2009wm} or \textbf{minimally almost periodic} \cite{neumann1934almost} groups in the infinite setting.  Gowers in the finite setting \cite{gowers2008quasirandom} and Bergelson and Furstenberg in the infinite setting \cite{bergelson2009wm} showed that, when amenable, such groups behave pseudorandomly in the sense that all of their large subsets contain roughly the expected number of finite product sets of a given length, and moreover, that it is easy to `switch' between any large sets via multiplication in the sense that if $A,B\subset G$ then there are roughly the expected number of $a\in A$ and $b\in B$ with $ab\in A.$

Inspired by this, we will call a coloring of an arbitrary amenable group \textbf{quasirandom} if every color class has roughly the expected number of finite product sets of each length (see Section \ref{sect:quasi.color} for precise definitions) and deduce that in such colorings it is very easy to switch between color classes via multiplication in a sense similar to before.  It turns out this latter property is precisely what is needed to upgrade the color focusing proof that deals with the configuration $\{x,xy,yx\}$ to also control the color of the $y$ term.  This is explained in Section \ref{sect:quasi}, where we prove Theorem \ref{thm:main} for quasirandom colorings of amenable groups (and, consequently, arbitrary colorings of amenable quasirandom groups).  

From here we then need to reduce the proof of Theorem \ref{thm:main} to the case that the coloring is quasirandom.  This is done in Section \ref{sect:reduct}, where we show that in an arbitrary coloring $c$ of an amenable group $G$ there is a $D\subset G$ such that $c$ is a quasirandom coloring of $G\setminus D$ and $G\setminus D$ is invariant enough under multiplication to run the proof from Section \ref{sect:quasi}.

\subsection{Acknowledgements}

\hspace{2mm}\vspace{2mm}

We thank the anonymous referees for carefully reading the earlier drafts of this paper and providing many corrections and useful feedback on its exposition.

\subsection{AI disclosure}

\hspace{1mm}

The initial version of this article, which contained all of the mathematical ideas, was posted to the arxiv in 2024 and involved no use of AI.    We have used Chatgpt 5.6 to aid with revising the current version of it.  Namely, we used it for the following tasks:

\begin{itemize}
    \item Identifying and correcting minor local issues such as misnamed variables, indexing errors, etc.
    \item Routine editorial tasks, such as making elements of $G^{<\omega}$ boldface to distinguish them from elements of $G.$
    \item The formal definition of the ordering $\pi$ described in Definition \ref{orderings} was produced using the help of GPT.  This was based on the ordering already implicit in the constructions in the original draft and an informal description of Figure \ref{f3} we had given it.
    \item The current deduction of Claim \ref{main tec} from \ref{construct} is an edited version of the write up produced by GPT.  This write up in turn was based on the terse version of the same deduction from the original version of the article and a sketch of the $N=1$ case we had provided it. 
\end{itemize}

\section{Notation and preliminaries}

We use standard notation such as $[r]=\{1,...,r\}.$

Throughout the paper, by $FP(x_0,...,x_k)$ we mean the set of all non-repeating finite products of $x_0,...,x_k,$ with products written from left to right in the order of the given indices.  For example, $FP(x_0,x_1)=\{x_0,x_1,x_0x_1\},$ while $FP(x_1,x_0)=\{x_1,x_0,x_1x_0\}.$ 

If $S=(s_1,...,s_n),$ by $FP(S)$ we mean $FP(s_1,...,s_n).$  We will only use this notation in cases where the identity of the elements $s_1,...,s_n$ is not important, and typically not specify anything about $S$ besides its size.

By an $IP_n$ set we mean a set of the form $FP(S)$ for some ordered tuple $S\subset G$ with $|S|=n$.  Recall that the finite unions theorem \cite{graham1991ramsey} or the existence of idempotent ultrafilters \cite{hindman.strauss} imply the following well-known fact:

\begin{lemma}\label{hindman}
    Let $G$ be an infinite group.  For any $r,n\in \mathbb{N}$ there is an $n'\in \mathbb{N}$ such that any $r$-coloring of an $IP_{n'}$ set contains a monochromatic $IP_n$ subset.  
\end{lemma}

\begin{remark}
    Throughout our constructions we can always ensure that all of the terms in $IP_n$ sets that appear are distinct.  In particular, the ultrafilter proof of Hindman's theorem gives us infinitely many choices for the next term in any ${IP}_n$ set that we construct, and only finitely many of these may conflict with our previous choices.  This is even more explicit in the definition of large $\IP$ sets (see Definition \ref{def: large IP}), where we have not only an infinite set but a positive measure set of such choices.
\end{remark}

We will also use left and right versions of the $\prod$ notation.  Namely, if $F\subset \{1,...,n\}$ then by $\prodr_{i\in F}x_i$ we mean $x_{i_1}\cdot ...\cdot x_{i_{|F|}},$ where $i_1<i_2<...<i_{|F|},$ and the reverse for $\prodl.$  We take the convention that products taken over the empty set are the group identity.

\subsection{Amenability and recurrence}\label{sec:php}

\hspace{3mm}

Throughout this paper, we say a group is \textbf{amenable} if it admits a finitely additive probability measure $\mu$ that is $G$-invariant, i.e., $\mu(A)=\mu(gA)=\mu(Ag)$ for all $A\subset G$ and $g\in G.$  Our arguments will in fact work for amenable semigroups, although we do not emphasize this.  If one prefers to avoid the use of such measures (whose existence in infinite groups can depend on the axiom of choice), then up to passing to appropriate F{\o}lner subsequences at various points of the argument one could replace the use of the measures $\mu$ with densities along two-sided F{\o}lner sequences.   

The following `density pigeonhole principle' will be the primary tool that allows our proof to work for amenable groups and is just an elementary variant of the Poincar\'e recurrence theorem.  A weaker analogue holds when replacing the positive measure set with a piecewise syndetic set in a non-amenable group, but that weaker variant is only enough to find monochromatic sets of the form $\{x,xy,yx\},$ see Proposition \ref{pws php}.

\begin{lemma}\label{php}
    Suppose that $\mu(A)>\ep$ and $(y_0,...,y_k)$ is a sequence with $k$ large enough depending on $\ep$.  There are $i\leq j$ such that $y=y_i\cdot y_{i+1}\cdot...\cdot y_j$ satisfies $\mu(A\cap Ay^{-1})>\ep^2/2.$  Similarly, there are $i'\leq j'$ such that $y'=y_{i'}\cdot...\cdot y_{j'}$ satisfies $\mu((y')^{-1}A\cap A)>\ep^2/2.$
\end{lemma}

This lemma is well-known, and its proof can be found in Bergelson's survey article \cite{bergelson2000multifarious} as Proposition 2.2.  



We will also need the following iterated version of the above fact, which follows immediately from repeated applications of Lemma \ref{php}.  The $r=1$ case is also recorded in \cite{bergelson2000multifarious} as Proposition 2.5.

\begin{lemma}\label{php2}
    Suppose that $\mu(A)>\ep$ and that $r,k\in \mathbb{N}.$  There is a $k'\in\mathbb{N}$ depending only on $r,k$ and $\ep$ such that if $S_1',...,S_r'\subset G$ with $|S_i'|\geq k',$ there are sets $S_1,...,S_r$ with $|S_i|=k$ and $FP(S_i)\subset FP(S_i')$, and a set $B\subset A$ with
    $$\mu(B)>\frac{\ep^{2^{rk}}}{2^{2^{rk}-1}},$$
    such that
    $$s_1\cdot...\cdot s_rB\subset A$$
    whenever $s_i\in FP(S_i)\cup\{1\}$ for each $i\in [r].$  Similarly, one can instead ensure that $Bs_1\cdot...\cdot s_r\subset A$ with the same parameters.
\end{lemma}


\subsection{`Many' tuples}\label{sec:many}

\hspace{3mm}

In order for our inductive arguments to go through we will often need to find `many' tuples with a given property.  Since finitely additive probability measures do not satisfy Fubini's theorem, we must introduce the following notation to make this precise.

\begin{definition}\label{def: many} Let $G$ be an amenable group with invariant finitely additive probability measure $\mu$.

\vspace{2mm}

\begin{itemize}

    \item Given a subset $X\subseteq G^{k+1},$ we say that $X$ is \textbf{large} if $$\mu(\{x_0\in G: \mu(\{x_1\in G:...: \mu(\{x_k\in G: (x_0,...,x_k)\in X\})>0\}....\})>0\})>0.$$

    \vspace{2mm}

    \item Given a property $P:G^{k+1}\rightarrow \{0,1\}$ we say that there are \textbf{many} tuples $(x_0,...,x_k)\in G^{k+1}$ satisfying $P$ if $\{\mathbf{v}\in G^{k+1}: P(\mathbf{v})=1\}$ is large.

    \vspace{2mm}

    \item If there are many tuples $(x_0,...,x_k)$ with property $P,$ we say that $(x_0,...,x_i)$ is a \textbf{valid choice} for $P$ if there are many $(x_{i+1},...,x_k)$ with $P(x_0,...,x_k)=1.$

    \vspace{2mm}

    \item If there are many tuples $(x_0,...,x_k)$ with property $P$ and $(x_0,...,x_i)$ is a valid choice for $P$, we write $$X_{i+1}(x_0,...,x_i)=\{x_{i+1}\in G: (x_0,...,x_{i+1}) \textnormal{ is a valid choice for } P\}.$$
\end{itemize}

\end{definition}

For example, we will say that a set $A$ contains many sets of the form $\{x_0,x_1,x_0x_1\}$ if $\mu(\{x_0\in G: \mu(\{x_1\in G: x_0,x_1,x_0x_1\in A\})>0\})>0.$  In this case, $x_0$ is a valid choice with this property if $\mu(\{x_1: x_0,x_1,x_0x_1\in A\})>0,$ and given a valid choice of $x_0$ we have $X_1(x_0)=\{x_1: x_0,x_1,x_0x_1\in A\}.$

This notion of being large is partition regular in the following sense.

\begin{lemma}\label{many partition}
    Suppose that $X\subseteq G^k$ is large and that
    $$X=X_1\cup...\cup X_r.$$
    Then $X_i$ is large for some $i\in[r].$
\end{lemma}

\begin{proof}
    We proceed by induction on $k.$  When $k=1,$ this follows from finite additivity, since a finite union of measure zero sets has measure zero.  Now suppose that $k>1$ and that the result has been proven for subsets of $G^{k-1}.$  For $x\in G,$ let
    $$X(x)=\{(x_2,...,x_k):(x,x_2,...,x_k)\in X\},$$
    and similarly for $i\in [r]$ define $$X_i(x)=\{(x_2,...,x_k):(x,x_2,...,x_k)\in X_i\}.$$  The set of $x$ for which $X(x)$ is large has positive measure.  For every such $x,$ the inductive hypothesis applied to
    $$X(x)=X_1(x)\cup...\cup X_r(x)$$
    shows that $X_i(x)$ is large for some $i\in[r].$  Hence the sets
    $$A_i=\{x\in G:X_i(x)\textnormal{ is large}\}$$
    satisfy $\mu(\bigcup_{i\in [r]}A_i)=\mu(\{x:X(x) \textnormal{ is large}\})$, and therefore $\mu(A_i)>0$ for some $i.$  This is exactly the assertion that $X_i$ is large.
\end{proof}

We will also need the following lemma, which follows from the fact that $\mu$ is both left and right invariant.

\begin{lemma}\label{many block products}
    Suppose that $V\subseteq G^M$ is large, and let $B_1,...,B_k\subseteq[M]$ be nonempty pairwise disjoint sets with $\max B_i<\max B_j$ for $i<j$.  For each $i,$ let $p_i:B_i\rightarrow B_i$ be a permutation.  For each $i\in [k]$ and $\mathbf{v}=(v_1,...,v_M)\in V$ let $B_i(\mathbf{v})=\prodr_{b\in B_i}v_{p_i(b)}.$  The set of tuples
    $$\{(B_1(\mathbf{v}),...,B_k(\mathbf{v})):\mathbf{v}\in V\}$$
    is large.
\end{lemma}

\begin{proof}
    For $i\in [k],$ suppose that $b'=\max B_i$ and we have already chosen a valid choice of $(v_1,...,v_{b'-1}).$  Define $$V_{b'}(v_1,...,v_{b'-1})=\{v\in G: \textnormal{there are many } (v_{b'+1},...,v_M) \textnormal{ with } (v_1,...,v_{b'-1},v,...,v_M)\in V\}$$  This has positive measure, so by the left and right invariance of $\mu$ so does the set $$\{\prodr_{b\in B_i}v_{p_i(b)}: v_{b'}\in V_{b'}(v_1,...,v_{b'-1})\}.$$  

    Since $\max B_i<\max B_j,$ this implies that $\{(B_1(\mathbf{v}),...,B_k(\mathbf{v})):\mathbf{v}\in V\}$ is large.
\end{proof}

\subsection{Finite product trees} \label{sect:tree}

\hspace{3mm}

In this subsection we introduce finite product trees, which will be the main objects along which we induct throughout the paper.  

Given an ordinal $\alpha$, we denote by $G^{<\alpha}=\bigcup_{n<\alpha}G^n$ the set of all sequences of elements of $G$ of length less than $\alpha.$  When working with trees, we use boldface letters for elements of $G^{<\omega}.$  We will identify $\mathbf{v}=(v_1,...,v_n)\in G^{<\omega}$ with the product $L(\mathbf{v})=v_1\cdot...\cdot v_n$ via the function $L:G^{<\omega}\rightarrow G$.  Given $\mathbf{v}=(v_1,...,v_m)\in G^m$ and $\mathbf{w}=(w_1,...,w_n)\in G^n,$ by $\mathbf{v}\smallfrown \mathbf{w}$ we mean the sequence $(v_1,...,v_m,w_1,...,w_n)\in G^{m+n}.$

\begin{definition}
    A \textbf{finite product tree} is a downward closed subset $T\subseteq G^{<\omega},$ i.e., for all $n\geq 1$ if $\mathbf{v}=(v_1,...,v_{n+1})\in T$ then $(v_1,...,v_{n})\in T.$ 
\end{definition}

\begin{remark}
    The above definition, and our notation concerning trees used throughout the paper, align with what are sometimes called \textbf{set-theoretic trees}, see, e.g., \cite{tserunyan2016introduction} Definition 2.1.  One can naturally identify a set-theoretic tree $T$ with a graph $H=(V,E),$ where $V=T$ and $(\mathbf{v}, \mathbf{w})\in E$ if $\mathbf{w}=\mathbf{v}\smallfrown g$ for some $\mathbf{v},\mathbf{w}\in T$ and $g\in G.$  Here, vertices of $H$ can be thought of as being labeled by tuples from $G^{<\omega}$ while edges are labeled by elements of $G.$  Notice that if $\mathbf{w}=\mathbf{v}\smallfrown {g}$ for some $\mathbf{v},\mathbf{w}\in T$  and $g\in G$, then $L(\mathbf{w})=L(\mathbf{v})\cdot g,$ hence the name `finite product tree.'
\end{remark}

\begin{definition}
    We will fix the following notation regarding trees:

    \begin{itemize}
        \item $T$ has \textbf{root} $g\in G$ if $T\cap G^1=\{{g}\}$. 

        \item A vertex $\mathbf{v}\in T\cap G^i$ has \textbf{height} $i-1.$  The tree $T$ has \textbf{height} $m$ if $T\cap G^{m+1}\neq \emptyset$ while $T\cap G^{m'+1}=\emptyset$ for all $m'>m.$  

        \item Given $\mathbf{v}=(v_1,...,v_i)\in T$, the \textbf{children} of $\mathbf{v}$ are the elements of $T(\mathbf{v})=T(v_1,...,v_i)=\{g\in G: \mathbf{v}\smallfrown{g}\in T\}$.  Notice here that $T(\mathbf{v})\subset G$ and is typically not a subset of $T.$
        
        \item If $T$ has height $m$ and $|T(\mathbf{v})|=n$ for all $\mathbf{v}\in T\cap G^{<m+1}$ then we say that $T$ has \textbf{branching} $n.$

        \item Given $\mathbf{v}\in T,$ a \textbf{path} from $\mathbf{v}$ is a sequence $\mathbf{P}=(g_1,g_2,...,g_k)$ with $g_1\in T(\mathbf{v})$ and $g_i\in T(\mathbf{v}\smallfrown g_1\smallfrown...\smallfrown g_{i-1})$ for all $2\leq i\leq k.$  Notice that $\mathbf{v}\smallfrown \mathbf{P}\in T$ and that if $T$ has height $m$ and $\mathbf{v}\in T\cap G^i,$ then the set of all paths from $\mathbf{v}$ together with root $L(\mathbf{v})$ is a finite product tree of height $m-i+1.$  A path in $T$ is simply a path from some vertex $\mathbf{v}\in T$.

        \item  Given $\mathbf{v}=(v_1,...,v_i)\in T$ and a path $\mathbf{P}$ from $\mathbf{v},$ we can \textbf{concatenate} $\mathbf{v}$ and $\mathbf{P}$, forming a new tree $T',$ by replacing $\mathbf{v}$ with $\mathbf{v}'=(v_1,...,v_i\cdot L(\mathbf{P})).$  The children of $\mathbf{v}'\in T'$ are the children of $\mathbf{v}\smallfrown \mathbf{P}$ in T.   

        \item A \textbf{subtree} of $T$ is any finite product tree that can be obtained from $T$ by removing or concatenating paths.  We use the notation $T\subseteq U$ to denote that $T$ is a subtree of $U,$ an abuse of notation.   

        \item We will use the notation $L(T)=\{L(\mathbf{v}): \mathbf{v}\in T\}.$
    \end{itemize}
\end{definition}

\section{Warm-up: the two-color case}

In this section we prove the $N=1$, $r=2$ case of Theorem \ref{thm:main}.  The argument here does not contain all of the ideas necessary to extend to more colors, but it was the starting point from which the proof of the more general result was discovered and captures the gist of the full argument while requiring significantly less notation and specialized arguments.  In particular, we will be able to entirely avoid the notations involving trees that are introduced in Sections \ref{sect:tree} and \ref{sect:quasi.color}. 

\begin{prop}\label{2 color}
    Let $G$ be an infinite amenable group and suppose that $G=C_0\cup C_1$ is a two-coloring.  Then there is a $y\in G$ and an $i\in \{0,1\}$ such that $\mu(\{x\in G: \{x,y,xy,yx\}\subset C_i\})>0$.
\end{prop}

The argument we give here can be easily adapted to guarantee that there are many $y$ satisfying the above and to obtain quantitative versions that also apply to finite groups, but we will not concern ourselves with this for ease of exposition.

Besides the variants of Poincar\'e's theorem discussed in Section \ref{sec:php}, our only tool in this section will be the following corollary of Hindman's finite products theorem:

\begin{lemma}\label{2 hindman}
    Let $G$ be an infinite group, and suppose that $G=C_0\cup C_1.$  There is an $i\in \{0,1\}$ such that for any $k\in\mathbb{N}$ there is a set $S\subset G$ with $|S|=k$ and $FP(S)\subset C_i.$
\end{lemma}

\begin{proof}[Proof of Proposition \ref{2 color}]
    By Lemma \ref{2 hindman}, we find $i\in \{0,1\}$ such that for any $k\in \mathbb{N}$ there is a set $S$ with $|S|=k$ and $FP(S)\subset C_i$.  Without loss of generality, let us assume that $i=1.$

    \begin{claim}\label{triv}
        If $\mu(C_0)=0,$ then $C_1$ contains monochromatic sets $\{x,y,xy,yx\}.$
    \end{claim}

    \begin{proof}
        By assumption, we have that $\mu(C_1)=1.$  Since $\mu$ is $G$-invariant, for any $y\in C_1$ we have that $\mu(C_1)=\mu(y^{-1}C_1)=\mu(C_1y^{-1}).$  Consequently, $\mu( C_1\cap y^{-1}C_1\cap C_1y^{-1})=1,$ and hence $\{x,y,xy,yx\}\subset C_1$ for all $x$ in this intersection.
    \end{proof}

    There are now two cases, the first of which will be generalized to more colors in Section \ref{sec:quasi}, and the second in Section \ref{sect:reduct}.

    \vspace{3mm}

    \textbf{Case 1:}\label{case 1} For some $\ep>0,$ $\mu(\{y_0\in C_0: \mu(y_0^{-1}C_0\cap C_0)>\ep\})>0.$

    \vspace{2mm}

    In this case, let $Y=\{y_0\in C_0: \mu(y_0^{-1}C_0\cap C_0)>\ep\},$ which, by hypothesis, has positive measure.  The result will now follow from the following claim and an easy color focusing argument. 

    \begin{figure}
\centering
    \includegraphics[scale=.3]{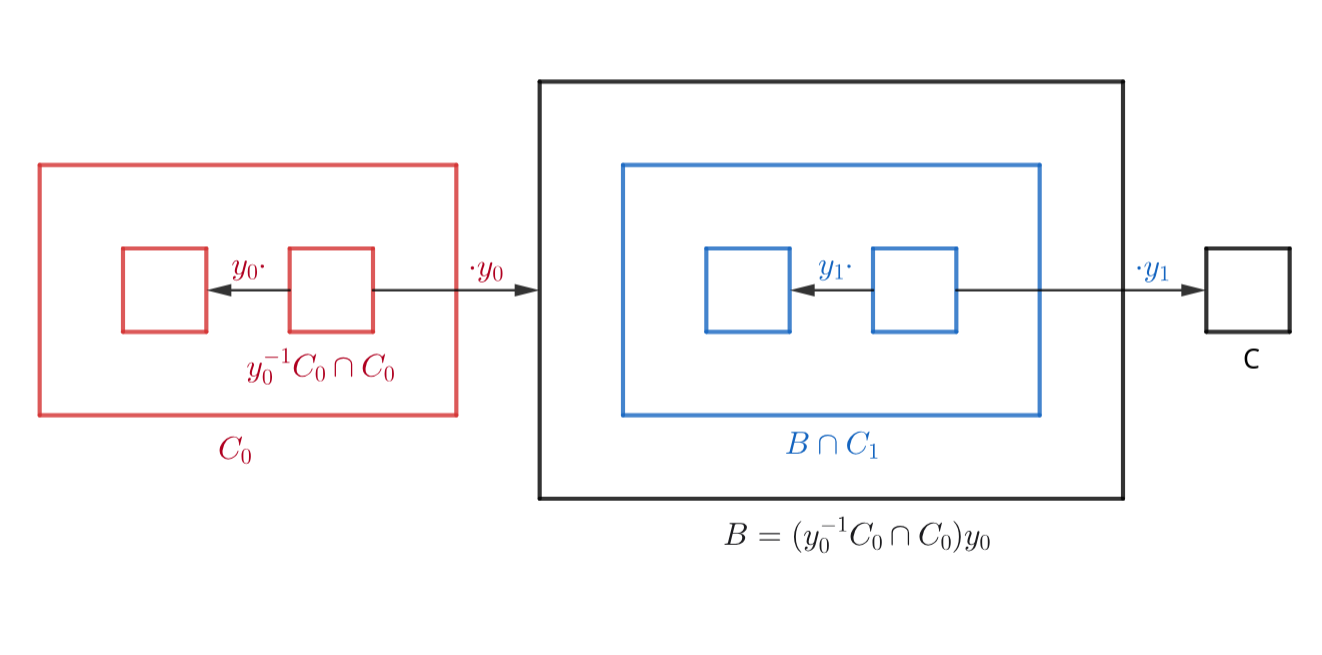}
    \caption{The color focusing step in the proof of Proposition \ref{2 color}, Case 1.  Each box represents a positive measure set.  The elements $y_0$ and $y_1$ are chosen to ensure that $y_0\in C_0$, $y_1\in C_1$, and $y_0y_1\in C_0$.  If Box $C$ contains an element $x'\in C_1$, then we can set $x=x'y_1^{-1}\in C_1 $ and set $y=y_1,$ ensuring that $\{x, y, xy, yx\}\subset C_1$.  If Box $C$ contains an element $x'\in C_0$ then we can set $x=x'y_1^{-1}y_0^{-1}\in C_0$ and set $y=y_0y_1,$ ensuring that $\{x, y, xy, yx\}\subset C_0$.}
    \label{f1}
\end{figure}
    
    \begin{claim}\label{2 color php}
        For any $r\in \mathbb{N},$ there is a $y_0\in Y$ and a set $S$ with $FP(S)\subset C_1$ and $|S|=r$ such that $y_0\cdot y_1\in C_0$ for any $y_1\in FP(S).$ 
    \end{claim}
    
    \begin{proof}
        Let $k'$ be large enough to apply Lemma \ref{php2} with set $A=Y$ and parameters $r=r$ and $k=1$.  By assumption, there is a set $S'$ with $|S'|=k'r$ such that $FP(S')\subset C_1.$  Let $S_1$ be the first $k'$ elements of $S',$ $S_2$ the next $k'$ elements, etc.  In particular, we now have disjoint tuples $S_1,...,S_r$ with $|S_i|=k'$ and $FP(\bigcup_{i=1}^r S_i)\subset C_1.$  Lemma \ref{php2} now ensures that there is a $y_0\in Y$ and, for each $i\in [r],$ an $s_i\in FP(S_i)$ with $y_0\cdot y_1\in Y\subset C_0$ for any $y_1\in FP((s_1,...,s_r)).$  Setting $S=(s_1,...,s_r)$ is as desired.
    \end{proof}

    Now let $r$ be large enough depending on $\ep/2$ to apply Lemma \ref{php} and let $y_0$ and $S$ be as guaranteed by Claim \ref{2 color php}.  Let $c_0=0$ and $A=y_0^{-1}C_0\cap C_0,$ which has measure at least $\ep$ as $y_0\in Y.$  By the pigeonhole principle, there exists a $c_1\in \{0,1\}$ such that $\mu(\{a\in A: ay_0\in C_{c_1}\})>\ep/2$.  Let $A'$ be the set of such $a.$  By our choice of $r$ and Lemma \ref{php}, we can find $y_1\in FP(S)$ such that $\mu(y_1^{-1}A'\cap A')>0$.  By the pigeonhole principle, there is a $c_2\in \{0,1\}$ such that $\mu(\{a\in y_1^{-1}A'\cap A': ay_0y_1\in C_{c_2}\})>0$.  Define $A_0$ to be the set of such $a,$ and let $A_i=A_0y_0\cdot...\cdot y_{i-1}$ for $i\in\{1,2\}$.

    We now finish with a simple color focusing argument.  See in particular Figure \ref{f1}.

    Since we have only two colors, we may choose $(i,j)\in \{(0,1),(1,2),(0,2)\}$ such that $c_i=c_j=c$ and $y=\prod_{k=i}^{j-1}y_k\in C_c$: if $c_1=0$, take $(i,j)=(0,1)$; if $c_1=1$ and $c_2=1$, take $(i,j)=(1,2)$; and if $c_1=1$ and $c_2=0$, take $(i,j)=(0,2)$.  Now consider $x\in A_i\subset C_{c}$.  Finally, $xy\in C_c$ as $c_j=c_i$ and $yx\in C_c$ by the definitions of $A,$ $A',$ and $A_0.$

    \vspace{3mm}

    \textbf{Case 2:}\label{case 2} For any $\ep>0,$ we have $\mu(\{y_0\in C_0: \mu(y_0^{-1}C_0\cap C_0)>\ep\})=0.$

    \vspace{3mm}

    \begin{figure}
\centering
    \includegraphics[scale=.3]{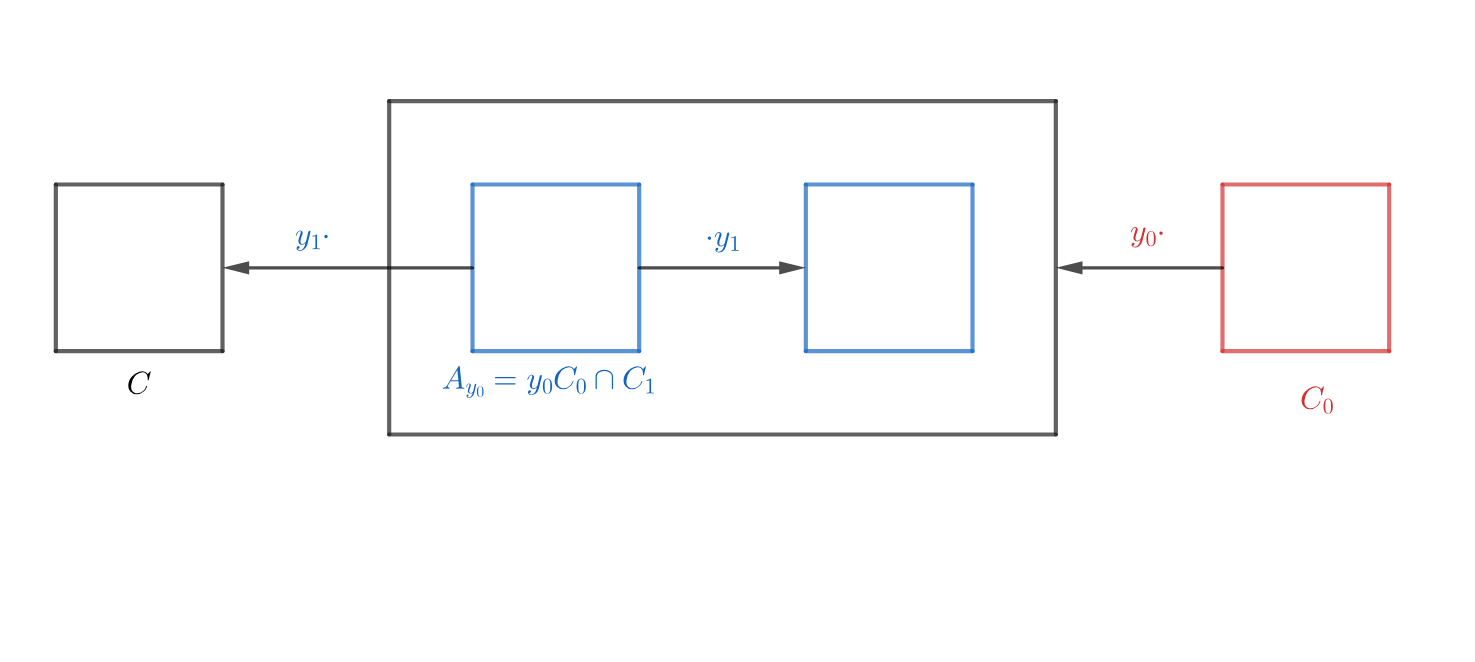}
    \caption{The proof of Proposition \ref{2 color} Case 2.  We choose $y_0\in Y'\subseteq C_0$ and $y_1\in C_1$ such that $y_1y_0\in C_0$.  The hypothesis of Case 2 ensures that $A_{y_0}$ has large measure.  If Box $C$ contains an element $x'\in C_1,$ we can set $x=y_1^{-1}x'$ and $y=y_1$ to find $\{x,y,xy,yx\}\subset C_1$.  Otherwise, $C_0$ has large measure which contradicts our assumption that $y_1y_0\in Y$.}
    \label{fig:2}
\end{figure}

    Define $\ep=\mu(C_0),$ which is greater than $0$ by Claim \ref{triv}.  Choose $\delta>0$ such that $(\ep-2\delta)^2/2>\delta.$  Let $$Y=\{y_0\in C_0: \mu(y_0^{-1}C_0\cap C_0)\leq\delta\}. $$ By the assumption of Case 2, we know that $\mu(Y)=\mu(C_0)=\ep.$  Let $k$ be large enough to apply Lemma \ref{php} to any set of measure greater than $\ep-2\delta$.  
    
    Applying Lemma \ref{php2} with $Y$ and $C_1,$ we find $Y'\subset Y$ and a set $S$ of size $k$ with $\mu(Y')>0,$ $FP(S)\subset C_1,$ and $y_1Y'\subset Y\subset C_0$ for any $y_1\in FP(S).$  For each $y_0\in Y',$ let $A_{y_0}=y_0C_0\cap C_1$.  Note that $\mu(A_{y_0})\geq\ep-\delta>\ep-2\delta$ by the $G$-invariance of $\mu$ and since $y_0\in Y$. Consequently, we may apply Lemma \ref{php} to find a $y_1\in FP(S)$ such that $A_{y_0}'=A_{y_0}\cap A_{y_0}y_1^{-1}$ has measure greater than $(\ep-2\delta)^2/2.$  
    
    Let $B_{y_0}=y_1A_{y_0}'$ and notice that if $\mu(B_{y_0}\cap C_1)>0,$ for any $x'\in B_{y_0}\cap C_1$ we have $x=y_1^{-1}x'\in C_1,$ $y=y_1\in C_1,$ $yx=x'\in C_1,$ and finally $xy\in A_{y_0}'\cdot y_1\subset A_{y_0}\subset C_1.$  This is exactly what we need to finish the proof, and so we may assume that $\mu(B_{y_0}\cap C_0)=\mu(B_{y_0})>(\ep-2\delta)^2/2.$

    In particular, we have that for each $y_0\in Y'$ there is a $y_1\in FP(S)$ such that $\mu({(y_1y_0)}^{-1}C_0\cap C_0)>(\ep-2\delta)^2/2>\delta.$  Since $FP(S)$ is finite, by the pigeonhole principle there is an $s\in FP(S)$ and a set $Y_s\subset Y'$ of positive measure such that $\mu({(sy')}^{-1}C_0\cap C_0)>\delta$ for all $y'\in Y_s.$  But $sY_s\subset C_0$ by our choice of $S,$ and $\mu(sY_s)>0,$ so this contradicts the assumption of Case 2.

\end{proof}

\section{Quasirandom colorings}\label{sect:quasi.color}

In this section we introduce quasirandom colorings and describe some of their basic properties.  In the next definition and throughout one should refer to Section \ref{sec:many} for our definition of the term `many.'

\begin{definition}\label{def: large IP}
    Given an amenable group $(G,\mu)$ we say that $A\subset G$ is a \textbf{large} $\IP$ set if for any $k\in \mathbb{N}$ there are many sequences $(x_0,...,x_k)$ with $FP(x_k,...,x_0)\subset A.$ 
    
    We say that a coloring of $G$ is \textbf{quasirandom} if every positive measure color class is a large $\IP$ set.
\end{definition}

\begin{remark}
    It is a consequence of \cite{bergelson2009wm} Theorem 1 that any finite coloring of a weak mixing group is quasirandom.  Up to introducing appropriate bounds, it follows from \cite{gowers2008quasirandom} Theorem 5.3 that any finite coloring of a $D$-quasirandom group is approximately quasirandom.  
\end{remark}

An easy consequence of the existence of idempotent ultrafilters is that any finite coloring of an amenable group contains a large $\IP$ color class.  Namely, if we consider the semigroup $(G,*),$ where $a*b=ba\in G$ and $p$ is an idempotent ultrafilter in the left ideal of $(\beta G,*)$ consisting of the ultrafilters $q$ with $\mu(A)>0$ for all $A\in q,$ then every element of $p$ is a large $\IP$ set.  In particular, every finite coloring of $G$ contains a color which is a large $\IP$ set.  See \cite{hindman.strauss} Section 6.7 for more detail.

We will not use the aforementioned ultrafilter facts.  Instead, the following lemma will suffice for our purposes, which is an elementary variant of what we just described.  Up to some setup, it follows from Hindman's finite products theorem, Lemma \ref{hindman}.

\begin{lemma}\label{pr}
    Large $\IP$ sets are partition regular, i.e., if $A=A_1\cup...\cup A_r$ and $A$ is a large $\IP$ set, then some $A_i$ is a large $\IP$ set.
\end{lemma}

\begin{proof}
    Fix $k'$ and let $k$ be large enough such that any $r$-coloring of an $IP_{k+1}$ set contains an $IP_{k'+1}$ set.  By assumption, $A$ contains many tuples $(x_0,...,x_k)$ with $FP(x_k,...,x_0)\subset A.$  
    
    For a fixed valid choice of $x_0,...,x_{k-1},$ consider $X_k=X_k(x_0,...,x_{k-1}).$ This satisfies $\mu(X_k)>0$ by assumption. Consider the auxiliary coloring $c:X_k\rightarrow [r^{2^{k+1}-1}]$ given by the string listing the colors of elements of $FP(x_k,...,x_0)$.  By the pigeonhole principle there is an $X_k'\subset X_k$ with $\mu(X_k')>\mu(X_k)/r^{2^{k+1}-1}$ that is monochromatic with respect to $c.$ Replacing $X_k(x_0,...,x_{k-1})$ with $X_{k}'$ gives many tuples $(x_0,...,x_k)$ with $FP(x_k,...,x_0)\subset A$ and so that for each $F\subset \{0,...,k-1\}$ and valid choice of  $x_0,...,x_{k-1}$ the color of $x_k\prodl_{i\in F}x_i$ depends only on $F$ and not on the choice of $x_k\in X_{k}'.$

    Now consider a fixed valid choice of $x_0,...,x_{k-2},$ and consider the auxiliary $r^{2^k}$ coloring of $X_{k-1}(x_0,...,x_{k-2})$ based on the value of $c(x_k)$ for some (and by the previous paragraph, any) $x_k\in X_k(x_0,...,x_{k-1})$.  By the pigeonhole principle we can find $X_{k-1}'\subset X_{k-1}(x_0,...,x_{k-2})$ with $\mu(X_{k-1}')\geq \mu(X_{k-1}(x_0,...,x_{k-2}))/r^{2^k}$  such that replacing $X_{k-1}(x_0,...,x_{k-2})$ with $X_{k-1}'$ gives many tuples $(x_0,...,x_k)$ with $FP(x_k,...,x_0)\subset A$ and so that for each $F\subset \{0,...,k\}$ and valid choice of  $x_0,...,x_{k-2}$ the color of $\prodl_{i\in F}x_i$ depends only on $F$ and not on the choice of $x_{k-1}\in X_{k-1}'$ or $x_k\in X_{k}(x_0,...,x_{k-1}).$

    Continuing in this way, we find many tuples $(x_0,...,x_k)$ such that $FP(x_k,...,x_0)\subset A$ and for each $F\subseteq \{0,...,k\}$ and valid choice of $x_0,...,x_k,$ the color of $\prodl_{i\in F}x_i$ depends only on $F$ and not on the choice of $x_0,...,x_k.$  Finally, applying the finite products theorem, Lemma \ref{hindman}, we find some $i$ and many choices of $(x'_0,...,x'_{k'})$ such that $FP(x'_{k'},...,x'_{0})\subset A_i.$  
    
    Repeating this for each $k'\in \mathbb{N}$, some fixed $i$ occurs for arbitrarily large values of $k'.$  Passing to smaller finite product sets then shows that $A_i$ is a large $\IP$ set.
\end{proof}

\subsection{Mixing in quasirandom colorings}

\hspace{3mm}

In this subsection we show that it is easy to `switch' between any desired color classes in a quasirandom coloring, in the sense that if $C_1,C_2$ are two positive measure color classes, then there are many $c_1\in C_1$ and $c_2\in C_2$ with $c_2c_1\in C_2$.  This is the primary tool that will allow our color focusing arguments to control the color of the step size $y$ in Theorem \ref{thm:main}, rather than just the $\{x,xy,yx\}$ terms that one would control in a more standard color focusing argument (such as the proof of Proposition \ref{prop: f2}).

This `switching' property can be seen as a qualitative analogue of the mixing property for quasirandom groups (see \cite{gowers2008quasirandom} Theorem 5.3 and \cite{bergelson2009wm} Theorem 2.3), which implies that one can switch between any finite collection of large sets via multiplication in that setting.

To be more precise, we will need the following definition.  Refer to Subsection \ref{sect:tree} for our notation involving trees.

\begin{definition}\label{cst}
Given sets $A,R_1,...,R_r\subset G,$  a \textbf{color switching tree} of height $m$ and branching $(2^n-1)r$ is a finite product tree $T$ with root $a\in A$ such that:

\begin{enumerate}
    \item for each $\mathbf{v}\in T\cap G^{<m+1},$ there are sets $S_1(\mathbf{v}),...,S_r(\mathbf{v})$ with $|S_i(\mathbf{v})|=n$ and $FP(S_i(\mathbf{v}))\subset R_i.$  The children of $\mathbf{v}$ are precisely the elements of $FP(S_i(\mathbf{v}))$ for each $i.$
    \item For any path $\mathbf{P}=(g_s,...,g_t)$ in $T$, if $g_s\in R_i$ then $L(\mathbf{P})\in R_i$.
    \item $L(\mathbf{v})\in A$ for all $\mathbf{v}\in T.$
\end{enumerate}

A \textbf{rootless color switching tree} is the same as above, but with the set $A=\{1\}$ and without condition (3).
\end{definition}

\begin{remark}
    Typically in our applications, the sets $R_1,...,R_r$ will be large $\IP$ subsets of distinct color classes.  In this case, condition (2) produces for any $i<j$ elements $g_i\in R_i$ and $g_j\in R_j$ with $g_i\cdot g_j\in R_i,$ and similarly for more terms.  In other words, we can switch between large $\IP$ color classes via multiplication, hence the name `color switching tree.' 
\end{remark}

For example, a color switching tree of height one and branching $(2^n-1)r$ consists of an element $a\in A$ and sets $FP(S_1),...,FP(S_r)\subset G$ with $|S_i|=n$  and such that $FP(S_i)\subset R_i$ and $a\cdot FP(S_i)\subset A$ for all $i\in [r].$  A rootless color switching tree is just the sets $FP(S_i)$ as before, with no set $A$ and element $a\in A$ specified. If $\mu(A)>0$ and each $R_i$ is a large $\IP$ set, then such height one trees exist by Lemma \ref{php2}.  

In Lemmas \ref{tree 1}-\ref{rootless}, we show that if $R_1,...,R_r$ are large $\IP$ sets, then for any set $A$ of positive measure one can construct a color switching tree of any desired height. We will in fact need to find many color switching trees in the sense of Definition \ref{def: many}, which will require some additional notation.  In order to make this precise, we need to fix an ordering $\pi$ on the elements of the tree so that we may describe each color switching tree as an element of $G^k$ for some $k$.  We point out that this is only needed to produce a large set of terms $y$ satisfying Theorem \ref{thm:main}.  If one is happy to find a single $y$ then this definition can be skipped; it is only used in Lemmas \ref{uniform pruning} and \ref{many paths}.

We will formally define this ordering as follows.  The recursive definition is somewhat cumbersome, but we are simply defining the tree from its leaves in, which is the order that the tree will be constructed.  See Figure \ref{f3} for a more intuitive illustration of the ordering.

\begin{figure}
\centering
    \includegraphics[scale=.3]{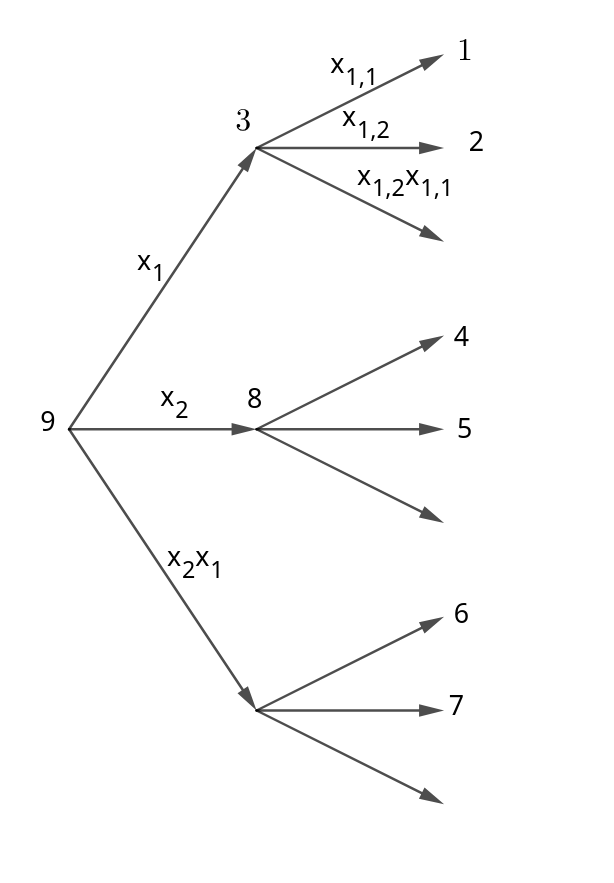}
    \caption{A rooted color switching tree with $r=1$ and $n=2$, as well as the ordering of terms $\pi$ indicating the order in which the tree is constructed.  Note that the vertices corresponding to products such as $x_2x_1$ do not receive labels as they are already determined by $x_1$ and $x_2$.  Similarly, the vertex corresponding to multiplying by $x_2$ is chosen after the children of the eventual $x_2x_1$ vertex are chosen.  }
    \label{f3}
\end{figure}

At every non leaf vertex $\mathbf{v}$ consider the set $S_i(\mathbf{v})$ as in Definition \ref{cst} (1).  Write
$$S_i(\mathbf{v})=(s_{i,n}(\mathbf{v}),...,s_{i,1}(\mathbf{v})),$$
so that the generators are chosen in the order $s_{i,1}(\mathbf{v}),...,s_{i,n}(\mathbf{v}).$  We use the colexicographic order on subsets: for distinct $F,F'\subseteq[j],$ $F$ is before $F'$ if the largest element of $F\mathbin{\triangle}F'$ belongs to $F'.$  For $1\leq i\leq r,$ $1\leq j\leq n,$ and $F\subseteq[j-1],$ set
$$g_{i,j,F}(\mathbf{v})=s_{i,j}(\mathbf{v})\prodl_{a\in F}s_{i,a}(\mathbf{v}).$$
These are precisely the children of $\mathbf{v}$ belonging to $FP(S_i(\mathbf{v})),$ grouped according to the largest index of a generator which occurs in the product.

\begin{definition}\label{orderings}
Given a color switching tree $T$ we define the parameter list $\pi(T)$ recursively.  If $T$ has height zero, let $\pi(T)$ be the empty tuple.  If $T$ is rootless and has positive height, then, for each child $g=g_{i,j,F}(1),$ let $T(g)$ be the rootless tree formed by the descendants of $g,$ with the initial edge $g$ removed.  List the subsets of $[j-1]$ in colexicographic order as $F_1,...,F_{2^{j-1}},$ and set
$$\pi_{i,j}(T)=\pi(T(g_{i,j,F_1}(1)))\smallfrown...\smallfrown\pi(T(g_{i,j,F_{2^{j-1}}}(1)))\smallfrown s_{i,j}(1).$$
We then define $\pi(T)$ by concatenating the tuples $\pi_{i,j}(T)$ in lexicographic order on the pairs $(i,j).$  Thus, before choosing $s_{i,j}(1),$ we choose the descendant trees below every child which becomes determined by that generator.  For a rooted tree with root $a,$ we append $a$ to the parameter list of the corresponding rootless tree.    
\end{definition}

A collection of rootless color switching trees of one fixed height and branching is said to be \textbf{large} if the corresponding set of tuples $\pi(T)$ is large in the sense of Definition \ref{def: many}.  A collection of rooted trees is large if the corresponding set of tuples $\pi(T)$ is large.  We say there are \textbf{many} color switching trees of a given height and branching if the set of such trees is large.


We show that given large $\IP$ sets $R_1,...,R_r$ and a positive measure set $A$ one can construct a large set of color switching trees of any desired height and branching.  This will proceed in three steps.  First, in Lemma \ref{tree 1} we argue that given a rootless tree of a fixed height and sufficiently large branching we can obtain a rooted tree of the same height and smaller branching by iterative applications of Lemma \ref{php2}.  Next, in Lemma \ref{uniform pruning} we argue that this construction can be made uniform, obtaining a large set of rooted trees from a large set of rootless trees.  Finally, the preceding two steps are then iteratively applied in Lemma \ref{rootless} to construct many trees of every height.

\begin{lemma}\label{tree 1}
    Suppose that $A,R_1,...,R_r\subset G$ with $\mu(A)>\varepsilon>0$.  For any $m,n$, there is an $n'$ depending on $m,n,r,\varepsilon$ such that if $U$ is a rootless color switching tree of height $m$ and branching $(2^{n'}-1)r$, then there is a rootless color switching subtree $T\subset U$ of height $m$ and branching $(2^n-1)r$ and a positive measure set $B\subseteq A$ such that $BL(T)\subseteq A$.  Moreover, $\mu(B)$ is bounded from below by a positive function depending only on $m,n,r$ and $\varepsilon$.  
\end{lemma}

\begin{proof}
    We proceed by induction on the height of the tree $m$.  For $m=1,$ consider $S_1,...,S_r$ with $|S_i|=n'$ and $FP(S_i)\subset R_i.$  So long as $n'$ is large enough with respect to $\varepsilon,$ we can apply Lemma \ref{php2} to the $IP_{n'}$ sets used in the definition of the tree $U$, obtaining the desired $IP_n$ subsets of $U$ and positive measure subset of $A.$

    Now suppose that we have proven the Lemma for trees of height $m-1$ and any branching and next wish to prove it for trees of height $m.$  Fix $n'$ large enough depending on $m,n,r$ and $\varepsilon,$ and let $U$ be the rootless color switching tree of height $m$ and branching $(2^{n'}-1)r.$  Let $S_i=S_i(1)\subset R_i$ be the sequence of sets corresponding to the children of $1$ in $U.$  Note, by definition, that each element $f\in FP(S_i)$ is the root of a color switching tree $U(f)$ of height $m-1,$ where for $U(f)$ we use parameters $A=R_i$ and $R_j=R_j$ for $j\in [r].$    As in the base case, applying Lemma \ref{php2} to $A$ and the set $S_i$ produces a set $A_i'\subset A$ of measure bounded from below and an $IP_n$ set $S_i'\subset FP(S_i)$ such that $A_i'\cdot f\subset A$ for each $f\in S_i'.$  For each such $f\in S_i'$, iteratively applying the inductive hypothesis to the sets $A_i'\cdot f\subset A$ and the rootless tree corresponding to the descendants of $f,$ we obtain $A_i\subset A$ with measure bounded from below and a subtree $T'(f)\subset U(f)$ of height $m-1$ and branching $(2^n-1)r$ such that $A_i\cdot t\subset A$ for each $t\in L(T'(f))$.  
    
    Iteratively running this procedure for each $i\in [r],$ we find sets $A\supset A_1\supset ... \supset A_r$ of measure bounded from below such that each element of $A_r$ forms the root for a color switching subtree $T\subset U$ of height $m$ and branching $(2^n-1)r.$ 
\end{proof}

\begin{lemma}\label{uniform pruning}
    Suppose that $A,R_1,...,R_r\subseteq G,$ let $\ep>0,$ and fix $m,n\in\mathbb{N}.$  Let $n_0$ and $\delta$ be supplied by Lemma \ref{tree 1} for these parameters.  Suppose that $\mu(A)>\ep,$ that $n'\geq n_0,$ and that there are many rootless color switching trees $U$ of height $m$ and branching $(2^{n'}-1)r.$  Then there are many rootless color switching trees $T$ of height $m$ and branching $(2^n-1)r$ such that a positive measure set of elements in $A$ can be a root for $T$.
    
\end{lemma}

\begin{proof}
    For a fixed rootless color switching tree $U$ of height $m$ and branching $(2^{n'}-1)r$, let $T(U)$ be the rootless color switching subtree with positive measure set of roots $A(T(U))$ provided by Lemma \ref{tree 1}.  
   
    Write
    $$\pi(U)=(z_1,...,z_M)\qquad\textnormal{and}\qquad \pi(T(U))=(t_1,...,t_K).$$
    Note from the construction in Lemma \ref{tree 1} that there are pairwise disjoint nonempty sets $B_1(U),...,B_K(U)\subseteq[M]$ and permutations $p_j(U):B_j(U)\rightarrow B_j(U)$ such that $t_j=\prodr_{b\in B_j(U)}z_{p_j(U)(b)}$ and
    $$\max B_1(U)<\max B_2(U)<...<\max B_K(U).$$

    Color the tree $U$ based on $B_1(U),...,B_K(U)$ and $p_1(U),...,p_K(U)$.  Since this is a finite coloring and there are many rootless color switching trees of height $m$ and branching $(2^{n'}-1)r$ by assumption, Lemma \ref{many partition} produces many such trees with the same color class.  Lemma \ref{many block products} now shows that the set of tuples $\pi(T(U))$ for trees $U$ in this color class is large, and for each such tree there is a positive measure set of roots $A(T(U))$ as needed. 
    

\end{proof}

\begin{lemma}\label{switch}\label{rootless}
    Suppose that $R_1,...,R_r\subset G$ are large $\IP$ sets.  For any $m,n\in \mathbb{N},$ there are many rootless color switching trees $T$ of height $m$ and branching $(2^n-1)r$.  Moreover, if $A$ is a positive measure set, there are many rooted color switching trees of any height and branching whose root is in $A$.
    
\end{lemma}

\begin{proof}
    We proceed by induction on $m$.  When $m=1,$ for each $n$ the existence of many rootless color switching trees of height $1$ and branching $(2^n-1)r$ is just the fact that, for each $c\in [r],$ there are many $n$-tuples $S$ with $FP(S)\subset R_c.$  This holds because each $R_c$ is a large $\IP$ set.

    Now suppose that we have proven the claim for fixed $m$ and any choice of $n'.$  To prove it for $m+1$ and $n$ we proceed as follows.  Fix a $c\in [r]$.  As $R_c$ is a large $\IP$ set, there are many tuples $(x_1,...,x_n)$ such that, letting $S=(x_n,...,x_1),$ $FP(S)\subset R_c$.  We will construct the branches of the rootless tree of height $m+1$ that start with color $c$ as in the ordering $\pi$ described after Definition \ref{cst}.  Repeating this for every color class $c$ gives the desired choices of color switching trees of height $m+1$. 
    
    Towards this, let $$X_1=\{x_1: \textnormal{ there are many tuples} (x_2,...,x_n) \textnormal{ with }FP(x_n,...,x_1)\subset R_c\}.$$  Note that $\mu(X_1)=\varepsilon>0$ by assumption.  Let $n_0$ be the parameter corresponding to $n$ and this choice of $\varepsilon$ and $m$ in Lemma \ref{uniform pruning}.  By the inductive hypothesis, there is a large set $\mathcal{U}_1$ of rootless color switching trees of height $m$ and branching $r(2^{n_0}-1).$  By Lemma \ref{uniform pruning}, there is a large set $\mathcal{T}_1$ of rootless color switching trees of height $m$ and branching $r(2^n-1)$ such that for each $T_1\in \mathcal{T}_1$ there is a positive measure subset $X_1'(T_1)\subset X_1$ of elements that can form the root of $T_1,$ as desired.

    We now repeat the argument for the remaining $n.$  Suppose that we have already constructed a choice of elements $x_1,...,x_{k-1}$ with $FP(x_{k-1},...,x_1)\subset R_c$ and their corresponding choices of trees, and that $$X_k=\{x_k: \textnormal{ there are many tuples} (x_{k+1},...,x_n) \textnormal{ with }FP(x_n,...,x_1)\subset R_c\}.$$

    Suppose that the colexicographical order of subsets of $[k-1]$ is $F_1,...,F_{2^{k-1}}.$  We will construct a decreasing sequence $X_k=X_{k,0}\supseteq X_{k,1}\supseteq ...\supseteq X_{k,2^{k-1}}$ of positive measure sets and choices of rootless color switching trees $T_{F_i}$ such that for any $x\in X_{k,i}$, one has $x\prodl_{f\in F_i}x_fL(T_{F_i})\subset R_c$ as needed.

    Suppose that we have constructed $X_{k,i-1}$ and let $y=\prodl_{f\in F_i}x_f$.  Since $\mu$ is $G$-invariant, we have that $\mu(X_{k,i-1}y)=\mu(X_{k,i-1})>0,$ so we may apply Lemma \ref{uniform pruning} as above to find a large set of rootless color switching trees $\mathcal{T}_{k,i}$ such that for each $T\in \mathcal{T}_{k,i}$ there is a positive measure subset $X(T)\subset X_{k,i-1}y$ such that each element of $X(T)$ can form a root for $T$.  Fix a choice of $T_{F_i}\in \mathcal{T}_{k,i}$ and set $X_{k,i}=X(T_{F_i})y^{-1}.$

    Once we have done this for all subsets $F_i,$ we may choose $x_{k}\in X_{k,2^{k-1}}$. 



    Finally, the moreover part now follows from applying Lemma \ref{uniform pruning} with a large set of rootless trees of the desired height and sufficiently large branching.

\end{proof}

Finally, recall from the start of this subsection that in our color focusing arguments we will choose the elements $x_1,...,x_N$ in Theorem \ref{thm:main} from consecutive paths in a color switching tree.  The next lemma ensures that being able to produce one choice of $x_1,...,x_N$ from a given tree together with the existence of many trees will allow us to construct many choices for $x_N,...,x_1.$

\begin{lemma}\label{many paths}
    Suppose that $\mathcal{T}$ is a large set of color switching trees of fixed height and branching, and that for each tree $T\in\mathcal{T}$ there is a designated sequence of paths $P_1(T),...,P_N(T)$ from $T$ such that $P_1(T)\smallfrown...\smallfrown P_i(T)$ is a path in $T$ for all $i.$  Then the set $$\{(L(P_N(T)),...,L(P_1(T))): T\in \mathcal{T}\} $$ is large. 
\end{lemma}

\begin{proof}
This follows from the same argument as in Lemma \ref{uniform pruning}: we may finitely color the set $\{\pi(T): T\in \mathcal{T}\}$ by the indices of the paths $P_i(T)$.  Applying Lemma \ref{many partition} we see that a large subset of $\mathcal{T}$ has the same chosen paths $P_i(T),$ and applying Lemma \ref{many block products} ensures that the set $\{(L(P_N(T)),...,L(P_1(T))): T\in \mathcal{T}\}$ is large.      
\end{proof}

\section{Proof of Theorem \ref{thm:main}}

In this section we prove Theorem \ref{thm:main}.  As mentioned previously, the proof will proceed in two steps.  The first step, presented in Section \ref{sec:quasi}, deals with quasirandom colorings and is a fairly natural extension of the argument that worked for two-colorings given in Proposition \ref{2 color} Case 1.  In Section \ref{sect:reduct} we reduce the case of general colorings to studying the case where the coloring is quasirandom.  This reduction has a similar idea to the proof of Proposition \ref{2 color} Case 2, although the details are somewhat more complicated.  

\subsection{The quasirandom case}\label{sec:quasi}\label{sect:quasi}

\hspace{3mm}

In this subsection we prove Theorem \ref{thm:main} for quasirandom colorings (and, consequently, for all colorings of quasirandom amenable groups).  The idea of the proof is similar to the `locally balanced' case of the proof of the main result from \cite{bowen2025monochromatic} and the `thick' case of the proof of the main result from \cite{bowen.sabok}, with the additional ease over the $+,\cdot$ results coming from the fact that the left and right actions of $G$ on $(G,\mu)$ commute.  Most of the work in the proof of Theorem \ref{thm:main} will come in the next subsection, where we reduce the problem to the quasirandom case studied here.


\begin{lemma}\label{proof:quasi}
    Suppose that $G=C_0\cup...\cup C_r$ is a quasirandom coloring.  For any $N\in \mathbb{N}$ there are many tuples $(x_N,x_{N-1},...,x_0)$ such that the following set is monochromatic: 
    
    $$\{x_i\cdot x_{i+1}\cdot...\cdot x_j, \overrightarrow{\prod_{a\in I}}x_a\cdot x_0\cdot x_1\cdot...\cdot x_j: 0\leq i\leq j\leq N, I\subseteq \{j+1,...,N\}\}.$$
\end{lemma}

\begin{proof}
    Before we begin, we note that the case $N=1,$ where one only desires monochromatic $\{x,y,xy,yx\},$ already has all of the needed ideas for the proof of the general result.  The only difference is that one applies the pigeonhole principle at the last step to find a monochromatic sequence of length $N+1$ rather than length $2.$

   Since the coloring is quasirandom, by definition, we may write the coloring as $G=C_0\cup C_1\cup...\cup C_r,$ where $C_0$ is the union of the color classes with measure $0$ (and hence has measure $0$ itself) and each $C_i$ for $i>0$ is a large $\IP$ set.
   
    Without loss of generality we may assume that $\mu(C_1)\geq \frac{1}{r}.$  Let $h:\mathbb{R}\rightarrow\mathbb{R}$ be the function $h(\alpha)=\alpha^2/(2r),$ and $h^i$ the composition of $h$ with itself $i$ times.  Let $n$ be large enough by Lemma \ref{php} to ensure that if $A$ is a set of measure at least $h^{N(r+1)}(\frac{1}{r})$ then for any $IP_n$ set $Y$ there is a $y\in Y$ with $\mu(y^{-1}A\cap A)>\mu(A)^2/2$.  Apply Lemma \ref{rootless} with large $\IP$ sets $R_i=C_i$, with $m=N(r+1),$ and $n$ as in the preceding sentence to construct many choices of rootless color switching trees $T$ of height $m$ and branching $(2^n-1)r.$  We will use $T$ to find a choice of $(x_N,...,x_0)$ satisfying the conclusion of Lemma \ref{proof:quasi}, and from Lemma \ref{many paths} deduce that there are many such choices.

    \begin{claim}\label{construct quasi}
        There are sequences $A_1=C_1\supset A_2\supset...\supset A_{N(r+1)}$ and $y_1,...,y_{N(r+1)-1}\in G$ and a function $f:[N(r+1)]\rightarrow [r]$ such that:
\vspace{4mm}
    \begin{enumerate}
        \item $\mu(A_{i+1})>\frac{\mu(A_i)^2}{2r}.$
        \item $y_iA_{i+1}\subset A_i$ and $A_{i+1}\subset A_i$.
        \item $A_{i+1}y_1\cdot...\cdot y_i\subset C_{f(i+1)}.$
        \item $y_1\in T(1)\cap C_1$ and $y_{i+1}\in T(1\smallfrown y_1\smallfrown...\smallfrown y_i)\cap C_{f(i+1)}.$
    \end{enumerate}
    \end{claim}

    \vspace{4mm}

    Before beginning the proof of Claim \ref{construct quasi} let us check that it allows us to finish the proof of Lemma \ref{proof:quasi}.  See Figure \ref{f1} for an illustration of this argument.  By the pigeonhole principle there are $a_1<a_2<...<a_{N+1}$ with $f(a_i)=f(a_j)=f$ for all $i,j\in [N+1]$.  For $1\leq i\leq N,$ let $x_i=y_{a_i}\cdot...\cdot y_{a_{i+1}-1}$. Let $X=A_{N(r+1)}\cdot y_1\cdot...\cdot y_{a_1-1}$ and let $x_0$ be any element of $X$.

    \vspace{2mm}
    
    Let us check that these satisfy the conclusion of Lemma \ref{proof:quasi}.  Notice that $X\subset C_{f(a_1)}=C_f$ by Items (2) and (3). Similarly, we have $x_i\cdot x_{i+1}\cdot...\cdot x_j\in C_{f(a_i)}=C_f$ for all $1\leq i\leq j\leq N$ by Item (4) and the fact that $T$ is a color switching tree. We have $\prodr_{a\in I}x_a\cdot X\cdot x_1\cdot...\cdot x_j\subset C_{f(a_{j+1})}=C_f$ for all $I\subset \{j+1,...,N\}$ by Items (2) and (3).  Moreover, by Item (1) we know $\mu(X)>0,$ and so for a given $T$ we find a single sequence $x_N,...,x_1$ and many valid choices of $x_0\in X$ satisfying the conclusion of the Lemma.  Finally, since Lemma \ref{rootless} gives us many choices of $T$ we have many choices of $x_N,...,x_1$ by Lemma \ref{many paths} as desired. 

    \vspace{2mm}

Therefore, it suffices to prove Claim \ref{construct quasi}.

\begin{proof}
    We begin by defining $A_1=C_1$ and $f(1)=1$.  Suppose that we have defined $A_1\supset ...\supset A_i$ and $y_{1},...,y_{i-1}$ and the partial function $f: \{1,...,i\}\rightarrow \{1,...,r\}$ satisfying Items (1)-(4) of Claim \ref{construct quasi}.  To continue, we apply Lemma \ref{php} with $A=A_i$ and $S=T(1\smallfrown y_1\smallfrown...\smallfrown y_{i-1})\cap C_{f(i)}$ to find $A_{i+1}'\subset A_i$ and $y_i\in S$ such that $y_iA_{i+1}'\subset A_i.$  Note that this is possible, since $\mu(A_i)>h^i(\frac{1}{r})>h^{N(r+1)}(\frac{1}{r})$ by Item (1) of the construction and by our choice of $n.$ Finally, by the pigeonhole principle there is some color $j$ such that 
$$\mu(\{x\in A_{i+1}': xy_1\cdot...\cdot y_i\in C_j\})\geq\frac{\mu(A_{i+1}')}{r}>\frac{\mu(A_i)^2}{2r}.$$

Letting $A_{i+1}$ be this subset of $A_{i+1}'$ and $f(i+1)=j$ is as desired.
\end{proof}
    
\end{proof}

\subsection{Reduction to the quasirandom case}\label{sect:reduct}

\hspace{3mm}

In this subsection we prove our main result, Theorem \ref{thm:main}.  In the previous subsection, we saw that it holds for quasirandom colorings, and so our goal in this section is to show that, in an arbitrary coloring, one can always pass to an `invariant enough' subset of $G$ where the coloring appears quasirandom and we can run the proof from Section \ref{sec:quasi} on this subset.  

The basic idea behind our argument is similar to the one used in Proposition \ref{2 color} Case 2, although the presence of more colors makes the argument more challenging.  Most notably, we must deal with the added complications that arise from only being able to assume that $C_0$ does not have many sets of the form $FP(x_k,...,x_0)$ for some $k.$

\begin{lemma}\label{tec}
    Suppose that $G=C_0\cup...\cup C_r$ where $C_i$ is a large $\IP$ set for $i>0$. For any $N\in \mathbb{N}$ either $C_0$ is also a large $\IP$ set or there is a $c\in \{1,...,r\}$ and many tuples $(x_N,...,x_0)$ with 

    $$\{x_i\cdot x_{i+1}\cdot...\cdot x_j, \overrightarrow{\prod_{a\in I}}x_a\cdot x_0\cdot x_1\cdot...\cdot x_j: 0\leq i\leq j\leq N, I\subseteq \{j+1,...,N\}\}\subset C_c.$$

\end{lemma}

Note that in the $N=1$ case our goal is to show that either $C_0$ is a large $\IP$ set or else there are many pairs $(y,x)$ with $\{x,y,xy,yx\}$ monochromatic in some color $c>0$.  One can keep this special case in mind during the proof; the arguments needed for it and the general case are the same except in Claim \ref{main tec}, where the $N>1$ argument is only slightly harder.

Before proving Lemma \ref{tec}, let us first note that together with Lemma \ref{proof:quasi} it implies Theorem \ref{thm:main}.

\begin{proof}[Proof of Theorem \ref{thm:main}]
    If every positive measure color class is a large $\IP$ set then we are done by Lemma \ref{proof:quasi}.  Otherwise, we may assume that we are given a coloring $G=C_0'\cup...\cup C_i'\cup C_{i+1}'\cup ...\cup C_{r'}',$ where $C_j'$ is a large $\IP$ set if and only if $j>i\geq 0.$  By Lemma \ref{pr} we may assume that $i<r'$ and that $C_0=C_0'\cup...\cup C_i'$ is not a large $\IP$ set.  Set $r=r'-i$ and define $C_j=C'_{i+j}$ for all $0< j\leq r$, and note that $C_j$ is a large $\IP$ set by assumption.  The coloring $G=C_0\cup C_1\cup...\cup C_r$ now satisfies the conditions of Lemma \ref{tec} and, since $C_0$ is not a large $\IP$ set, the conclusion ensures that we find the desired choices of $x_N,...,x_0$ in a color class $C_j=C'_{i+j}$ for some $j>0$.
\end{proof}

\begin{proof}[Proof of Lemma \ref{tec}]  
    For tuples $(x_N,...,x_0)$, write
    \begin{multline*}
    \mathcal{P}(x_N,...,x_0)=\\
    \{x_i\cdot x_{i+1}\cdot...\cdot x_j, \overrightarrow{\prod_{a\in I}}x_a\cdot x_0\cdot x_1\cdot...\cdot x_j: 0\leq i\leq j\leq N, I\subseteq \{j+1,...,N\}\}.
    \end{multline*}
    Assume that no $C_c$ for $c>0$ contains many tuples $(x_N,...,x_0)$ with $\mathcal{P}(x_N,...,x_0)\subset C_c.$

    In order to prove Lemma \ref{tec}, we therefore need to deduce that $C_0$ is a large $\IP$ set.  We may also assume that $\mu(C_0)>0,$ or else we are done by Lemma \ref{proof:quasi}.

    For any prescribed height and branching, we may choose a rootless color switching tree $U$ such that, for every $c>0$ and every sequence of nonempty paths $\mathbf{P}_1,...,\mathbf{P}_N$ with $\mathbf{P}_1\smallfrown...\smallfrown\mathbf{P}_N$ a path in $U,$ we have

    \begin{equation}\label{U choice}
        \mu(\{x_0\in G:\mathcal{P}(L(\mathbf{P}_N),...,L(\mathbf{P}_1),x_0)\subset C_c\})=0.
    \end{equation}

    \vspace{1mm}
    
    Indeed, Lemma \ref{rootless} gives many rootless color switching trees of the prescribed height and branching.  If every such tree admitted paths and a color $c>0$ for which the displayed measure were positive, then Lemma \ref{many partition} would give a large family with the same color $c>0,$ and Lemma \ref{many paths} would give many choices of $(x_N,...,x_1)$ with a positive measure set of valid choices of $x_0$ such that $\mathcal{P}(x_N,...,x_0)\subset C_c$, contrary to our assumption on the color classes $C_c$ for $c>0$.

    We will show that $C_0$ is a large $\IP$ set by showing that for each $k$ there are many $(x_0,...,x_k)$ with $FP(x_k,...,x_0)\subset C_0.$  Towards this, in the arguments below, for a given $k$ we will choose a color switching tree $U$ of sufficiently large height and branching satisfying Equation \ref{U choice} and take all of our trees to be subtrees of this fixed tree.  Notice that the preceding property is inherited by subtrees, since consecutive paths in a subtree expand to consecutive paths in $U.$



    Our main tool for proving Lemma \ref{tec} will be the following technical claim.

    \begin{claim}\label{main tec}
        Let $C'\subset C_0$ have positive measure, let $m,n\in \mathbb{N},$ and let $F\subset G$ be finite and non-empty.  Fix a tree $U$ satisfying Equation \ref{U choice}.  There are $m',n'\in\mathbb{N}$ with the following property.  If $T'\subset U$ is a color switching tree of height $m'$ and branching $(2^{n'}-1)r$ such that $F\cdot L(T')\subset C_0,$ then there is a positive measure subset $A\subset C'$ and a rooted color switching subtree $T\subset T'$ of height $m$ and branching $(2^n-1)r$ such that $A\cdot F\cdot L(T)\subset C_0.$  Moreover, $m',n',$ and a positive lower bound on $\mu(A)$ depend only on $m,n,|F|,N,r,$ and a positive lower bound on $\mu(C').$ 
    \end{claim}

    Before proving Claim \ref{main tec}, let us first use it to complete the proof of Lemma \ref{tec}.  

\begin{proof}[End of the proof of Lemma \ref{tec} given Claim \ref{main tec}.]

\hspace{1mm}
\vspace{1mm}
    
    Fix the coloring from Lemma \ref{tec} and parameters $N$, $r$ and $\mu(C_0)$ from its statement.  Recall that in order to finish the proof we must show that $C_0$ is a large $\IP$ set.  We will in fact prove the following stronger property: 

    \vspace{1mm}
    
    For each $k,m,n\in \mathbb{N}$ there are $M,Q\in\mathbb{N}$ such that, given any rootless color switching tree $U(M,Q)$ satisfying Equation \ref{U choice} and with height $M$ and branching $(2^Q-1)r,$ we can construct for $i\leq k$ a color switching subtree $T_i\subset U(M,Q)$ of height $m$ and branching $(2^n-1)r$ and (many choices of) $(x_0,...,x_k)$ such that $$FP(x_ka_k,...,x_0a_0)\subset C_0$$ for all $a_i\in L(T_i).$ 

    \vspace{1mm}
    
    This stronger property will be needed in our induction.  Note that since the trees $T_i$ do not depend on the choices of $x_i$, fixing any sequence of $a_i\in L(T_i)$ and using the invariance of $\mu$ implies that there are many $(x_0',...,x_k')$ with $FP(x_k',...,x_0')\subset C_0$.  Doing this for every $k$ ensures that $C_0$ is a large $\IP$ set as desired.

\vspace{1mm}
    This construction will be via an induction on $k.$  

    \vspace{1mm}
    
    When $k=0$, take $M=m$ and $Q$ large enough to apply Lemma \ref{tree 1} with the set $C_0$ and the tree $U(M,Q)$.  This gives a rootless color switching subtree $T_0\subset U(M,Q)$ of height $m$ and branching $(2^n-1)r$ and a positive measure subset $C'\subseteq C_0$ all of whose elements can be a root of $T_0.$  In other words, $x_0a_0\in C_0$ for all $x_0\in C'$ and $a_0\in L(T_0),$ completing the base case. 

    Now suppose that the result has been proven for $k.$  Choose $m',n'$ sufficiently large for the successive applications of Claim \ref{main tec} below.  These applications of Claim \ref{main tec} will output a set $X'_{k+1}$ with a uniform lower bound $\mu(X'_{k+1})\geq \varepsilon$.  Choose $M,Q$ sufficiently large both to apply the inductive hypothesis with parameters $k,m',n'$ inside the tree $U(M,Q)$ and to apply Lemma \ref{tree 1} with a set of measure at least $\varepsilon$ and desired subtree height $m$ and branching $(2^n-1)r.$  By the inductive hypothesis, we find many choices of $(x_0,...,x_k)$ and color switching subtrees $T_0',...,T_k'\subset U(M,Q)$ of height $m'$ and branching $(2^{n'}-1)r$ and with 
    \begin{equation}\label{equation}
        FP(x_ka_k,...,x_0a_0)\subset C_0
    \end{equation}
    for all $a_i\in L(T_i').$  

    Our induction step proceeds via the following Sub-claim:

    \begin{subclaim}\label{sub claim construction}
        Given a valid choice of $(x_0,...,x_k)$ and $T_i'\subset U(M,Q)$ satisfying Equation \ref{equation}, there is a positive measure set $X_{k+1},$ a rootless color switching tree $T_{k+1}$ of height $m$ and branching $(2^{n}-1)r,$ and subtrees $T_i\subset T_i'$ of height $m$ and branching $(2^n-1)r$ such that $FP(x_{k+1}a_{k+1},...,x_0a_0)\subset C_0$ for all $a_i\in L(T_i)$ and $x_{k+1}\in X_{k+1}.$  Moreover, the tree $T_{k+1}$ will also be a subtree of $U(M,Q).$
    \end{subclaim}

    \begin{proof}
        The construction of the trees $T_i\subset T_i'$ for a given choice $(x_0,...,x_k)$ can be achieved by iterating Claim \ref{main tec}.  To see this, we first start with $X=C_0$ and apply Claim \ref{main tec} with $C'=X,$ $F=\{x_k\}$, and $T'=T_k'$ to find a positive measure set $X_1\subset X$ and $T_k\subset T_k'$ as in the Claim.  Thus
    $$X_1\cdot x_k\cdot L(T_k)\subset C_0.$$

    In other words, $FP(z,x_ka_k)\subset C_0$ for all $z\in X_1$ and $a_k\in L(T_k)$.
    
    We then repeat this, applying Claim \ref{main tec} with $C'=X_1$, $F=\{x_ka_kx_{k-1}: a_k\in L(T_k)\}\cup\{x_{k-1}\}$, and $T'=T_{k-1}'$ to find the needed positive measure set $X_2\subset X_1$ and $T_{k-1}\subset T_{k-1}'.$  Note that $F\cdot L(T_{k-1}')\subset C_0$ here by (\ref{equation}), and thus $FP(z,x_ka_k,x_{k-1}a_{k-1})\subset C_0$ for all $z\in X_2$ and $a_i\in L(T_i)$.  
    
    Continuing in this way, at the $i$th stage applying Claim \ref{main tec} with the finite set $F$ of all products of the already treated terms followed by $x_i,$ gives subtrees $T_i\subset T_i'$ for $i\leq k$ and after $k$ steps a positive measure set $X_{k+1}'$ such that
    $$FP(z,x_ka_k,...,x_0a_0)\subset C_0$$
    for all $a_i\in L(T_i)$ and all $z\in X_{k+1}'.$  To obtain the set $X_{k+1}$ and the rootless tree $T_{k+1},$ we now apply Lemma \ref{tree 1} to the set $X_{k+1}'$ and a truncation of the fixed tree $U_{(M,Q)}$ of height $m.$  This gives a positive measure set $X_{k+1}\subset X_{k+1}'$ and a rootless color switching subtree $T_{k+1}\subset U$ of height $m$ and branching $(2^n-1)r$ such that $x_{k+1}a_{k+1}\in X_{k+1}'$ for every $x_{k+1}\in X_{k+1}$ and $a_{k+1}\in L(T_{k+1}).$ 
    \end{proof}

    With Sub-claim \ref{sub claim construction} proven for a given valid choice of $(x_0,...,x_k)$, to finish the proof of Lemma \ref{tec} it only remains to make the choices of trees $T_i$ uniform across a large set of valid choices of $(x_0,...,x_k)$.  Call the sequence of trees $(T_0,...,T_{k+1})$ corresponding to a given choice $(x_0,...,x_k)$ its type.  Since all of these trees are subtrees of the fixed finite tree $U(M,Q)$ there are only finitely many types, so we can color $x_k\in X(x_0,...,x_{k-1})$ based on the type of $(x_0,...,x_k)$ and hence choose a positive measure monochromatic subset.  In other words, for a given choice of $(x_0,...,x_{k-1}),$ we can assume all valid choices of $x_k$ have the same type.  Repeating an analogous argument for the remaining coordinates, we find many choices of $(x_0,...,x_k)$ with the same type as desired.

\end{proof}

    The proof of Claim \ref{main tec} follows from running the proof of Lemma \ref{proof:quasi} simultaneously for the sets $C'f$ for each $f\in F.$  This can be accomplished by using a `color focusing tree,' a variant of the typical color focusing technique introduced in \cite{bowen.ramsey.tree} which takes advantage of the fact that one has many choices for each next step size in a typical color focusing argument.  Our choice of $U$ satisfying Equation \ref{U choice} ensures that each branch of the color focusing tree that we construct will terminate in color $C_0.$  The following Claim is the construction of this `color focusing tree'.

    \begin{claim}\label{construct}
        Provided $n'$ is sufficiently large in terms of $\mu(C'),m',n,r,$ and $|F|,$ there is a positive measure set $A\subset C'$ and a color switching subtree $T''\subset T'$ of height $m'$ and branching $(2^n-1)r$ such that:
    \begin{enumerate}

        \item $AfL(\mathbf{v})$ is monochromatic for each $f\in F$ and $\mathbf{v}\in T''.$

        \item For any rooted path $\mathbf{v}=(v_1,...,v_k)\in T'',$ any $f\in F,$ any $1\leq j\leq k,$ and any $I\subset \{j+1,...,k\},$ the set
        $$\prodr_{a\in I}v_a\cdot Af\cdot v_1\cdot...\cdot v_j$$
        is monochromatic of the same color as $Af\cdot v_1\cdot...\cdot v_j.$
    \end{enumerate}
    Moreover, $\mu(A)$ is bounded from below by a function depending only on $\mu(C'),m',n,r,$ and $|F|.$
    \end{claim}

    Before proving Claim \ref{construct}, we will first use it to deduce Claim \ref{main tec}.

    \begin{proof}[Proof of Claim \ref{main tec}]
        Choose $m'$ sufficiently large in terms of $m,N,r,$ and $|F|,$ and then choose $n'$ sufficiently large to apply Claim \ref{construct}, obtaining a tree $T''\subset T'$ of height $m'$ and branching $(2^n-1)r$.  For $\mathbf{v}\in T''$ and $f\in F,$ let $c_f(\mathbf{v})$ denote the color of the monochromatic set $AfL(\mathbf{v}).$

        The argument proceeds by using the following sub-claim, which is a variant of the color focusing argument used to deduce Lemma \ref{proof:quasi} from Claim \ref{construct quasi}.  The main idea is that since the tree $T''\subset T'\subset U$ satisfies Equation \ref{U choice}, the color focusing argument must always return color $0$.

        \begin{subclaim}\label{N=2 subclaim}
            There is no sequence of vertices $\mathbf{u}_0,\mathbf{u}_1,...,\mathbf{u}_N$ on a common branch of $T'',$ an $f\in F,$ and a $c>0$ such that $c_f(\mathbf{u}_i)=c$ for every $i,$ and such that the path from $\mathbf{u}_{i-1}$ to $\mathbf{u}_i$ begins with an element of $C_c$ for every $1\leq i\leq N.$  
        \end{subclaim}

        \begin{proof}
            Let $\mathbf{P}_i$ be the path from $\mathbf{u}_{i-1}$ to $\mathbf{u}_i$ and set $x_i=L(\mathbf{P}_i)$ for $1\leq i\leq N.$  Since $T''$ is a color switching tree,
        $$x_i\cdot...\cdot x_j\in C_c$$
        for all $1\leq i\leq j\leq N.$  Let $X=AfL(\mathbf{u}_0).$  Then $\mu(X)>0,$ $X\subset C_c,$ and
        $$Xx_1\cdot...\cdot x_j=AfL(\mathbf{u}_j)\subset C_c$$
        for all $0\leq j\leq N,$ where the product is empty when $j=0.$  Finally, if $I\subset \{j+1,...,N\},$ then $\prodr_{a\in I}x_a$ is a correctly oriented product of terms occurring after $\mathbf{u}_j,$ and Item (2) of Claim \ref{construct} gives
        $$\prodr_{a\in I}x_a\cdot X\cdot x_1\cdot...\cdot x_j\subset C_c.$$
        It follows that $\mathcal{P}(x_N,...,x_1,x_0)\subset C_c$ for every $x_0\in X,$ contrary to our choice of $U$ satisfying Equation \ref{U choice}.  Consequently, no such sequence of vertices exists.
        \end{proof}

        We now iteratively use Sub-claim \ref{N=2 subclaim} to produce the desired color switching tree.  Call a vertex $\mathbf{v}\in T''$ bad if $c_f(\mathbf{v})>0$ for some $f\in F.$  Start with an arbitrary child of the root.  At any stage, if the current vertex and the next $m$ levels contain no bad vertex then stop.  Otherwise, choose a bad vertex $\mathbf{u}$ in these levels, choose $f\in F$ and $c>0$ with $c_f(\mathbf{u})=c,$ pass to a child of $\mathbf{u}$ belonging to $C_c,$ and repeat.  If this procedure continued for more than $Nr|F|$ steps, then some pair $(f,c)$ would occur at $N+1$ of the chosen bad vertices.  The path from each such vertex to the next begins with an element of $C_c,$ contradicting Sub-Claim \ref{N=2 subclaim}.  Thus, after at most $Nr|F|$ steps, we reach a non-root vertex $\mathbf{w}$ such that no extension of $\mathbf{w}$ at distance at most $m$ is bad.  Taking $m'$ sufficiently large ensures that there are still at least $m$ levels below $\mathbf{w}.$

        We therefore have
        $$AfL(\mathbf{v})\subset C_0$$
        for every $f\in F$ and every vertex $\mathbf{v}$ extending $\mathbf{w}$ at distance at most $m$.  Retain these descendants of $\mathbf{w}$ and concatenate the path from the root of $T''$ to $\mathbf{w}$ into the new root.  Let $T$ be the resulting tree.  By the definition of a subtree, the children below the new root are unchanged, and hence $T$ has height $m$ and branching $(2^n-1)r.$  The color switching properties are also unchanged.  Thus $T$ is a rooted color switching subtree of $T'.$  Moreover,
        $$A\cdot F\cdot L(T)\subset C_0,$$
         completing the proof. 
    \end{proof}

    Consequently, to finish the proof it suffices to prove Claim \ref{construct}.  The argument is similar to that given in the proof of Claim \ref{construct quasi}, but we construct the entire tree rather than a single path.

    \begin{proof}[Proof of Claim \ref{construct}]
        Let $\mathbf{r}$ be the root of $T'.$  We will iteratively find $C'\supset A_0\supset A_1\supset...\supset A_i$ and $T'=T_0\supset T_1\supset ...\supset T_i$ such that $A_i$ and $T_i$ satisfy Properties (1) and (2) of Claim \ref{construct} for rooted paths extending $\mathbf{r}$ to height $i$, and terminate with $A=A_{m'}$ and $T''=T_{m'}$.  This is accomplished by iteratively applying the pigeonhole principle and Lemma \ref{php2}, with the only difficulty being the need to keep track of notation.
        
        To begin, we first find a set  $A_0\subset C'$ with $\mu(A_0)\geq \mu(C')/(r+1)^{|F|}$ such that each set $A_0fL(\mathbf{r})$ is monochromatic for $f\in F.$  To see that this is possible, consider the coloring of $x\in C'$ where $x$ is colored based on the $|F|$-tuple listing the colors of each element $xfL(\mathbf{r})$ for $f\in F$.  By the pigeonhole principle, there is a color class of measure at least $\mu(C')/(r+1)^{|F|},$ which we can take to be $A_0.$

        We now find a positive measure subset $A_1\subset A_0$ and a subtree $T_1\subset T_0$ such that $vA_1\subset A_0$ and the set $A_1fL(\mathbf{r})v$ is monochromatic for each $f\in F$ and $v\in T_1(\mathbf{r}).$  Towards this, recall from the definition of color switching trees that for each color $c\in \{1,...,r\}$ there is a set $S_c'$ of size $n'$ such that $FP(S_c')\subset C_c$ and $FP(S_c')\subset T_0(\mathbf{r})$.  As long as $n'$ is large enough with respect to $n,$ we can apply Lemma \ref{php2} to find for each color class $c\in [r]$ a tuple $S_c$ of size $n$ with $FP(S_c)\subset FP(S_c')$ and a positive measure subset $A_0'\subset A_0$ such that $vA_0'\subset A_0$ for all $v\in \bigcup_{c\in \{1,\ldots,r\}}FP(S_c).$  Let $T_1\subset T_0$ be the subtree obtained by discarding all paths that do not start with an element of $\bigcup_{c\in \{1,\ldots,r\}}FP(S_c).$  Notice that since $vA_0'\subset A_0$ and each set $A_0fL(\mathbf{r})$ is monochromatic for $f\in F,$ we also have that $vA_0'fL(\mathbf{r})$ is monochromatic of the same color as $A_0fL(\mathbf{r})$ for all $v\in T_1(\mathbf{r}),$ i.e., Property (2) of Claim \ref{construct} is satisfied at this stage.  
        
        Similar to the construction of $A_0\subset C'$ above, we now find a set $A_1\subset A_0'$ of measure bounded from below such that $A_1fL(\mathbf{r})v$ is monochromatic for all $f\in F$ and $v\in T_1(\mathbf{r}),$ i.e., Property (1) of Claim \ref{construct} is also satisfied.  More precisely, we color elements $x\in A_0'$ by the $r(2^n-1)|F|$-tuple listing the color of $xfL(\mathbf{r})v$ for all $f\in F$ and $v\in T_1(\mathbf{r})$ and choose $A_1$ to be the maximal measure color class.
        
        The remainder of the argument follows from iterating the above construction for all $m'$ levels of the tree.  Suppose we have found $A_0\supset A_1\supset...\supset A_i$ and $T'=T_0\supset T_1\supset ...\supset T_i$ with the desired properties, where $1\leq i<m'.$  To find $A_{i+1}$ and $T_{i+1},$ consider a rooted path $\mathbf{v}=\mathbf{r}\smallfrown(v_1,...,v_i)\in T_i.$  Applying Lemma \ref{php2} with $A_i$ and the generating tuples corresponding to the children $T_i(\mathbf{v})$, we find $A_i'\subset A_i$ of measure bounded from below and $T_{i+1,\mathbf{v}}\subset T_i$ such that Property (2) of Claim \ref{construct} is satisfied for all $f\in F$ and paths of the form $\mathbf{v}\smallfrown v_{i+1},$ where $v_{i+1}\in T_{i+1,\mathbf{v}}(\mathbf{v}).$  Applying the pigeonhole principle, we find $A_{i,\mathbf{v}}\subset A_i'$ of measure bounded from below such that the sets $A_{i,\mathbf{v}}fL(\mathbf{v})v_{i+1}$ are monochromatic for any choice of $v_{i+1}\in T_{i+1,\mathbf{v}}(\mathbf{v})$ and $f\in F,$ ensuring that we satisfy Property (1) of Claim \ref{construct} at this stage of the construction.  Iteratively repeating this process for each path $\mathbf{v}\in T_i\cap G^{i+1}$ gives the desired set $A_{i+1}=\bigcap_{\mathbf{v}\in T_i\cap G^{i+1}} A_{i,\mathbf{v}}$ and subtree $T_{i+1}=\bigcap_{\mathbf{v}\in T_i\cap G^{i+1}} T_{i+1,\mathbf{v}}.$
    \end{proof}

\end{proof}

    \section{The pattern $\{x,xy,yx\}$ in non-amenable groups}

    In this section we prove Proposition \ref{prop: f2}.  The idea is essentially the same as in the proof of Lemma \ref{proof:quasi}, but we use a weaker pigeonhole principle and do not need to keep track of the color of $y.$

    Before beginning, we recall a couple of needed definitions.

    \begin{definition}

    In a group $G:$
        \begin{itemize}
            \item A set $T\subset G$ is (right) \textbf{thick} if for any finite $F\subset G$ there is a $t\in G$ with $Ft\subset T.$
            \item A set $S\subset G$ is (left) \textbf{syndetic} if there is a finite set $F\subset G$ such that $F^{-1}S=G.$
            \item A set $A\subset G$ is \textbf{piecewise syndetic} if it is the intersection of a thick set and a syndetic set.  
        \end{itemize}
    \end{definition}

    We will also use the following basic facts about these notions of size.  Their proofs all follow quickly from the definitions and can be found in \cite{bergelson1998notions}, Section 2.

    \begin{prop}\label{fact}
        Let $G$ be a group and $A\subset G$ a piecewise syndetic set.

        \begin{enumerate}
            \item For every $g\in G,$ the sets $gA$ and $Ag$ are piecewise syndetic.
            \item If $A=C_0\cup...\cup C_r,$ then some $C_i$ is piecewise syndetic.
            \item If $A$ is piecewise syndetic and $A\subseteq B$ then $B$ is piecewise syndetic.
        \end{enumerate}
    \end{prop}

    We will also need the following variant of the pigeonhole principle.

    \begin{prop}\label{pws php}
        If $A\subset G$ is piecewise syndetic, then $\{g: g^{-1}A\cap A \textnormal{ is piecewise syndetic}\}$ is syndetic.
    \end{prop}

    This fact has a standard ultrafilter proof (which follows from combining \cite{hindman2011algebra} Theorems 4.39 and 4.40, for example) but we were unable to find a reference containing a combinatorial proof.  We give one below:

   \begin{proof}[Proof of Proposition \ref{pws php}]

    This will follow from the following auxiliary claim about thick sets.

    \begin{claim}
       If $T$ is thick, then for all $g\in G$ so is the set $g^{-1}T\cap T.$
    \end{claim}
       \begin{proof}
           Consider a finite subset $F\subset G$ and, using the fact that $T$ is thick, find a $t\in G$ such that $(gF\cup F)t\subset T.$  This implies that $Ft\subset g^{-1}T\cap T$ as desired.
       \end{proof}

       Now to prove the Proposition, consider a piecewise syndetic set $A=S\cap T,$ where $S$ is syndetic and $T$ is thick.  Since $S$ is syndetic, there is a finite set $F\subset G$ such that $F^{-1}S=G.$  Let
       $$T'=T\cap\bigcap_{f\in F}f^{-1}T.$$
       The set $T'$ is still thick: for a finite set $H\subset G,$ choose $t$ such that $(H\cup\bigcup_{f\in F}fH)t\subset T.$  Moreover, $T'\subset F^{-1}A.$  Indeed, for $x\in T'$ choose $f\in F$ with $fx\in S$; the definition of $T'$ also gives $fx\in T,$ and hence $fx\in A.$  Replacing $T$ by $T',$ by our previous claim we know $T(g)=g^{-1}T\cap T$ is thick, and so $A\cap T(g)=S\cap T\cap T(g)=S\cap T(g)$ is piecewise syndetic.  Moreover, for all $x\in A\cap T(g)$ there is an $f\in F$ such that $fgx\in A.$  Partitioning $A\cap T(g)$ according to a corresponding $f$ and applying Proposition \ref{fact} (2), we see that $(fg)^{-1}A\cap A$ is piecewise syndetic for some $f\in F.$  As this holds for all $g\in G,$ we see that $F^{-1}\{g: g^{-1}A\cap A \textnormal{ is piecewise syndetic}\}=G,$ i.e. the set is syndetic.
   \end{proof}

    Proposition \ref{prop: f2} now follows from a standard dynamical color focusing argument:

    \begin{proof}
        Without loss of generality, assume that $C_0$ is piecewise syndetic.  We inductively build sequences of sets $A_0=C_0\supset A_1\supset...\supset A_{r+1}$ and $S_0,...,S_r\subset G,$ a sequence $y_0,...,y_r\in G,$ and a function $f:\{-1,0,...,r\}\rightarrow \{0,...,r\}$ such that:

        \begin{enumerate}
            \item $A_i$ is piecewise syndetic.
            \item $y_iA_{i+1}\subset A_{i}$ for each $i\geq 0.$
            \item $f(-1)=0$
            \item $A_{i+1}y_0\cdot...\cdot y_i$ is monochromatic of color $f(i)$ for each $i\geq 0.$

        \end{enumerate}

    To build such a sequence and use it to complete the proof, suppose that we have already found $A_i,$ $y_0,...,y_{i-1},$ and a partial function $f$ as desired.  Note by Proposition \ref{pws php} that there is a syndetic set $S_i$ such that $A'_i(y)=y^{-1}A_i\cap A_i$ is piecewise syndetic for each $y\in S_i.$  Moreover, coloring $a\in A'_i(y)$ based on the color of $ay_0\cdot...\cdot y_{i-1}y$ and applying Proposition \ref{fact} (2), for each $y$ we find a piecewise syndetic subset $A_i(y)\subset A'_i(y)$ such that $A_i(y)y_0\cdot...\cdot y_{i-1}y$ is monochromatic.  Now color $y\in S_i$ based on the color of this monochromatic set and apply Proposition \ref{fact} (2) again to find a piecewise syndetic set $S\subset S_i$ on which this color is constant.  Let $f(i)$ be this color.  If $f(i)=f(j)$ for some $j<i,$ with products ending at $y_j$ understood to be empty when $j=-1,$ consider
    $$Y=y_{j+1}\cdot...\cdot y_{i-1}S.$$
    For $s\in S,$ set $y=y_{j+1}\cdot...\cdot y_{i-1}s$ and consider the piecewise syndetic set
    $$X_s=A_i(s)y_0\cdot...\cdot y_j.$$
    If $a\in A_i(s),$ then $x=ay_0\cdot...\cdot y_j$ and $yx$ belong to $A_{j+1}y_0\cdot...\cdot y_j,$ while $xy$ belongs to $A_i(s)y_0\cdot...\cdot y_{i-1}s.$  Thus $x,xy,$ and $yx$ all have color $f(j)=f(i).$  Since $Y$ and every $X_s$ are piecewise syndetic by Proposition \ref{fact} (1), this completes the proof.  Otherwise, we choose $y_i\in S$ arbitrarily and set $A_{i+1}=A_i(y_i)$ to continue with our construction.  As we must eventually have $f(i)=f(j)$ for some $j<i$ by the pigeonhole principle, this process returns the desired configuration after at most $r+1$ steps.

    \end{proof}

\section{Questions}

  We end with the following related questions, the first of which is a coloring variant of \cite{kra2022infinite} Question 8.7, which asked if every positive density set in an amenable group contains a subset of the form $AB\cup BA$ with $A,B$ infinite.

  \begin{question}
      Suppose that a sufficiently non-commutative group is finitely colored.  Are there infinite sets $A,B$ such that $AB\cup BA$ is monochromatic and $ab\neq ba$ for all $a\in A$ and $b\in B$?  What if we also assume that the group is amenable? 
  \end{question}

  A weaker form of the preceding question, in which $A$ is required to be infinite but $B$ is only required to have size at least $k$ for any fixed $k$, follows from the pigeonhole principle.  The proof of Proposition \ref{prop: f2} can also be adapted to give a somewhat stronger result, ensuring that we can find an infinite $A$ and an arbitrarily large finite $B$ with $A\cup AB\cup BA$ monochromatic.

Our second question concerns left vs right recurrence in amenable groups and requires the following definition.

\begin{definition}\label{def: rec}
    Let $G$ be an amenable semigroup with invariant measure $\mu.$  We say that $S\subset G:$

    \begin{itemize}
        \item is a \textbf{set of weak left recurrence} if $\{s\in S: s^{-1}A\cap A\neq \emptyset\}\neq \emptyset$ for all $A$ with $\mu(A)>0.$

        \item is a \textbf{set of left recurrence} if $\{s\in S: \mu(s^{-1}A\cap A)>0\}\neq \emptyset$ for all $A$ with $\mu(A)>0.$

        \item is a \textbf{set of nice left recurrence} if $\{s\in S: \mu(s^{-1}A\cap A)>\mu(A)^2-\delta\}\neq \emptyset$ for all $A$ with $\mu(A)>0$ and all $\delta>0.$
    \end{itemize}

    In all of the above items we say that $S$ is a \textbf{large set of (weak/nice) recurrence} if the set of such $s$ has positive measure rather than just being non-empty.  We also define right versions of these notions analogously.
\end{definition}

A crucial fact used during the proof of Theorem \ref{thm:main} was that left $IP_0$ sets and right $IP_0$ sets coincide, and, in particular via Lemma \ref{php}, are both left and right sets of nice recurrence.

\begin{question}
    When do the left and right notions from Definition \ref{def: rec} coincide?
\end{question}

\bibliographystyle{amsalpha}
\bibliography{bib}

\end{document}